\documentclass[a4paper,11pt,reqno]{amsart}
\usepackage{amsmath}
\usepackage{amsthm,enumerate}
\usepackage{mathtools}
\usepackage[normalem]{ulem}
\usepackage[T1]{fontenc}
\usepackage[utf8]{inputenc}
\usepackage[thinc]{esdiff}
\usepackage{color}
\usepackage{dsfont}
\usepackage{fancyhdr}
\usepackage{calc}
\usepackage{url}

\usepackage[text={6.0in,8.6in},centering]{geometry}

\usepackage[citecolor=blue,colorlinks]{hyperref}

\usepackage{xcolor}
\usepackage{xspace}

\theoremstyle{plain}
\newtheorem{theorem}{Theorem}[section]
\newtheorem{proposition}[theorem]{Proposition}
\newtheorem{lemma}[theorem]{Lemma}
\newtheorem{corollary}[theorem]{Corollary}

\theoremstyle{definition}
\newtheorem{definition}[theorem]{Definition}

\theoremstyle{remark}
\newtheorem{remark}[theorem]{Remark}
\newtheorem{example}[theorem]{Example}

\DeclareMathOperator\arctanh{arctanh}

\newcommand{\mm}{\mathfrak m}
\newcommand{\R}{\mathbb R}
\newcommand{\Per}{\mathrm{Per}}

\newcommand{\di}{\mathsf{d}}

\begin{document}
\title[
The equality case in Cheeger's and Buser's inequalities on $ \mathsf{RCD}$ spaces
]
{
The equality case in Cheeger's and Buser's inequalities on $ \mathsf{RCD}$ spaces
}
\author{Nicol\`o De Ponti}\thanks{Nicol\`o De Ponti: Scuola Internazionale Superiore di Studi Avanzati (SISSA), Trieste,  Italy, \\email: ndeponti@sissa.it} 
\author{Andrea Mondino}\thanks{Andrea Mondino: Mathematical Institute, University of Oxford, UK, \\email: Andrea.Mondino@maths.ox.ac.uk} 
\author{Daniele Semola}\thanks{Daniele Semola: Scuola Normale Superiore, Pisa, Italy, \\email: Daniele.Semola@maths.ox.ac.uk}

\begin{abstract}
We prove that the sharp Buser's inequality obtained in the framework of $\mathsf{RCD}(1,\infty)$ spaces by the first two authors  \cite{DeMo} is rigid, i.e. equality is obtained if and only if the space splits isomorphically a Gaussian. The result is new even in the smooth setting.   We also show that the equality in Cheeger's inequality is never attained in the setting of $\mathsf{RCD}(K,\infty)$ spaces with finite diameter or positive curvature, and we provide several examples of spaces with Ricci curvature bounded below where these assumptions are not satisfied and the equality is attained.
\end{abstract}

\maketitle
\section{Introduction}
In the paper we consider a complete and separable metric space $(X, \mathsf{d})$ endowed with a Borel measure $\mm$, finite on bounded sets. The triple $(X, \mathsf{d}, \mm)$ is called \emph{metric measure space}, m.m.s. for short.
The space of real-valued Lipschitz (resp. bounded Lipschitz, Lipschitz with bounded support, Lipschitz on bounded sets) functions over $X$ will be denoted by $\mathsf{Lip}(X)$ (resp. $\mathsf{Lip}_{b}(X)$, $\mathsf{Lip}_{bs}(X)$, $\mathsf{Lip}_{loc}(X)$). The slope of a function $f:X\rightarrow \R$ at $x\in X$ is defined by
\begin{equation}
\mathrm{lip}(f)(x):=\limsup_{y\rightarrow x} \frac{|f(y)-f(x)|}{\di(y,x)}\, ,
\end{equation}  
with the convention $\mathrm{lip}(f)(x)=0$ if $x$ is an isolated point.

We introduce the following relevant definitions: in case $\mm(X)<\infty$, for any $1<p<\infty$ we set
\begin{equation}\label{eq:defla1Intro}
\lambda_{1,p}(X)= \inf\bigg\{\frac{\int_X \mathrm{lip}(f)^p \di\mm}{\int_X |f|^p\,\di\mm}: \ 0\not\equiv f\in \mathsf{Lip}_{bs}(X), \int_X |f|^{p-2}f\, \di\mm=0\bigg\}.
\end{equation} 

If $\mm(X)=\infty$ and $1<p<\infty$ we instead set 
\begin{equation}\label{eq:defla0Intro}
\lambda_{0,p}(X)= \inf\bigg\{\frac{\int_X \mathrm{lip}(f)^p \di\mm}{\int_X |f|^p\,\di\mm}: \ 0\not\equiv f\in \mathsf{Lip}_{bs}(X)\bigg\}.
\end{equation} 
When there is no risk of confusion we will drop the dependence on the ambient space writing $\lambda_{0,p}$ and $\lambda_{1,p}$.  Moreover, the shorthand notation $\lambda_0,\lambda_{1}$ will be used to refer to $\lambda_{0,2}$, $\lambda_{1,2}$ respectively, when there is no risk of confusion.  Under quite general assumptions on the m.m.s. $(X,\di,\mm)$, the quantities $\lambda_0$ and $\lambda_1$ correspond to the first two eigenvalues of the Laplace operator (see \autoref{thm:compembedding}). 
\medskip

Let $A\subset X$ be a Borel set, the \textit{perimeter} $\mathrm{Per}(A)$ is defined as:
\begin{equation*}
\mathrm{Per}(A):=\inf\bigg\{\liminf_{n\rightarrow \infty}\int_X \mathrm{lip}(f_n)\,\di\mm: f_n\in \mathsf{Lip}_{loc}(X), f_n\rightarrow \chi_A \ \mathrm{in} \ L^1(X,\mm)\bigg\},
\end{equation*}
where we denote by $\chi_A:X\to\{0,1\}$ the indicator function of the set $A\subset X$.

The \textit{Cheeger constant} of the metric measure space $(X,\mathsf{d},\mm)$ is defined as follows: 
\begin{equation}\label{eq:defChConst}
h(X):=
\begin{cases}
\inf  \left\{\frac{\Per(A)}{\mm(A)}\, :\, A\subset X \text{ Borel with $0<\mm(A)\leq \mm(X)/2$} \right\} & \text {if } \mm(X)<\infty, \\
\inf  \left\{\frac{\Per(A)}{\mm(A)}\, :\, A\subset X \text{ Borel with $0<\mm(A)<\infty$} \right\} &\text {if } \mm(X)=\infty.
 \end{cases}
\end{equation}

In \cite{ChIn} Cheeger obtained the following celebrated inequality, now known as \emph{Cheeger's inequality}:
\begin{equation}\label{eq:ChIn}
\lambda_{1}\geq \frac{1}{4} h(X)^{2}.
\end{equation}

The original result of Cheeger was in the framework of smooth and compact Riemannian manifolds, but the argument of the proof is very robust, as noticed (even earlier) by Maz'ya in \cite{Maz} (see also \cite{Grig1}), and it can be extended to  more general  frameworks.  We refer to \cite[Appendix A]{DeMo} for a proof on general metric measure spaces.

When $X$ is a compact Riemannian manifold of dimension $n$ and Ricci curvature that satisfies ${\rm Ric}\geq K$, $K\leq 0$, Buser \cite{Buser} proved that also the following \emph{upper bound} for $\lambda_{1}$ in terms of $h(X)$ holds:
\begin{equation}\label{eq:BuserConst}
\lambda_1(X)\leq 2\sqrt{-(n-1)K}h(X)+10h(X)^2.
\end{equation}

Thanks to a result of Ledoux \cite{Ledoux}, we also know that the constants in Buser's inequality \eqref{eq:BuserConst} can be chosen to be dimension-independent. More precisely, Ledoux proved the following inequality for all smooth connected Riemannian manifolds of finite volume:
\begin{equation}\label{eq:LedouxConst}
\lambda_1(X)\leq \max\{6\sqrt{-K}h(X),36h(X)^2\}. 
\end{equation}

An improvement of Buser's inequality has been also noticed by Agol in the unpublished \cite{Ag}, where he refines the original proof of Buser and obtains better estimates of $\lambda_1$ in terms of $h(X)$ for $3$ dimensional manifolds, and later by Benson in \cite{Ben}.
\medskip

Recently, De Ponti and Mondino \cite{DeMo} sharpened the aforementioned theorems of Buser and Ledoux by improving the constants in both the Buser-type inequalities \eqref{eq:BuserConst}-\eqref{eq:LedouxConst} and by extending the results to (possibly non-smooth) $\mathsf{RCD}(K,\infty)$  spaces.  
\\ Recall that $\mathsf{RCD}(K,\infty)$ spaces are (possibly non-smooth) metric measure spaces having Ricci curvature bounded below by $K\in \R$ and no upper bound on the dimension, in a synthetic sense.  More precisely,  $\mathsf{RCD}(K,\infty)$ spaces  are the sub-class of $\mathsf{CD}(K,\infty)$ spaces introduced in the seminal works of Sturm \cite{St} and Lott-Villani \cite{LV} having the canonical energy functional  (called ``Cheeger energy") satisfying the parallelogram identity.  The reader is referred to Section \ref{sec:prel} for the precise definitions, and to \cite{AmbrosioICM}  for a survey. 
   The class of $\mathsf{RCD}(K,\infty)$ spaces was singled out by Ambrosio-Gigli-Savar\'e \cite{AGS1}  (see also \cite{AGMR}) who developed a powerful calculus in this setting. 
\\   The subclass of $\mathsf{RCD}(K,\infty)$ spaces having an upper bound on the dimension by $N\in [1,\infty)$ in a synthetic sense is denoted by $\mathsf{RCD}(K,N)$, see \cite{G, EKS, AMS, CaMi}.
\\  Remarkable examples of $\mathsf{RCD}(K,\infty)$ spaces are pmGH-limits of Riemannian manifolds with Ricci curvature bounded below (the so-called Ricci limits ) \cite{GMS}, finite dimensional Alexandrov spaces \cite{P}, weighted Riemannian manifolds with $\infty$-Bakry-\'Emery Ricci curvature bounded below by $K$ \cite{St}, stratified spaces \cite{BKMR}, (possibly singular) quotients of Riemannian manifolds with Ricci bounded below \cite{GGKMS}.
\medskip

In order to state the outcomes of \cite{DeMo}, we firstly set
\begin{equation}\label{eq:ExplJKIntro}
J_K(t):=\begin{cases}\sqrt{\frac{2}{\pi K}}\arctan\Big(\sqrt{e^{2Kt}-1}\Big)  \ \ &\textrm{if} \ \  K>0,\\
\frac{2}{\sqrt{\pi}}\sqrt{t} \ \ &\textrm{if} \ \ K=0,\\
\sqrt{-\frac{2}{\pi K}}\arctanh{\Big(\sqrt{1-e^{2Kt}}\Big)} \ \ &\textrm{if} \ \ K<0.\end{cases} \qquad t>0
\end{equation}

\begin{theorem}[Theorem 1.1 \cite{DeMo}]\label{thm:MainDeMo}
Let $(X,\mathsf{d},\mm)$ be an $\mathsf{RCD}(K,\infty)$ metric measure space for some $K\in \mathbb{R}$, with $\mm(X)<\infty$. Then
\begin{equation}\label{implicitlambda1Intro}
h(X)\geq \sup_{t>0}  \frac{1-e^{-\lambda_1t}}{J_K(t)}.
\end{equation}
\end{theorem}

We refer again to \cite{DeMo} for a discussion on how to obtain more explicit bounds of $\lambda_1$ in terms of $h(X)$ starting from the inequality \eqref{implicitlambda1Intro} (improving the constants in both \eqref{eq:BuserConst}-\eqref{eq:LedouxConst}), and for an analogous result that can be applied to spaces with $\mm(X)=\infty$ (in this case $\lambda_1$ is replaced by $\lambda_0$). 

As noticed in  \cite{DeMo}, another important consequence of \autoref{thm:MainDeMo} is that the inequality is sharp in the case $K>0$, as equality is achieved in the Gaussian space. 
\medskip

A first goal of the present work is to show that the inequality \eqref{implicitlambda1Intro} is also rigid:
\begin{theorem}\label{thm:Main}
Let $(X,\mathsf{d},\mm)$ be an $\mathsf{RCD}(K,\infty)$ metric measure space with $K>0$. Let us suppose that  
\begin{equation}\label{eq: equality Buser}
h(X)=\sup_{t>0}  \frac{1-e^{-\lambda_1t}}{J_K(t)}.
\end{equation}
Then 
$$(X,\di,\mm)\cong (Y,\di_Y,\mm_Y)\times (\R,|\cdot|,\sqrt{K/(2\pi)}e^{-Kt^2/2}\di t)$$
for some $\mathsf{RCD}(K,\infty)$ space $(Y,\di_Y,\mm_Y)$, where $\cong$ denotes isomorphism as metric measure spaces.
\end{theorem}

Let us stress that the rigidity result of \autoref{thm:Main} is new even in the smooth setting of  (possibly weighted) Riemannian manifolds. 
\\ Using the compactness of the class of $\mathsf{RCD}(K,N)$ spaces with uniformly bounded diameter under measured Gromov-Hausdorff convergence \cite{St1, GMS},  the stability properties of $\lambda_{1}$ and $h(X)$ under such convergence \cite{GMS, AH17}, and the fact that no  $\mathsf{RCD}(K,N)$ space can split isomorphically a Gaussian (since the former is measure-doubling while the latter is not), we obtain the next dimensional improvement of \eqref{implicitlambda1Intro} by a straightforward argument by contradiction (notice that $K>0$ implies a uniform upper bound on the diameter thanks to the Bonnet-Myers theorem \cite{St1}).

\begin{corollary}[Dimensional improvement of Buser's inequality]\label{cor:improvedBuser}
For every $K>0$ and $N\in [1,\infty)$ there exists $\varepsilon=\varepsilon(K,N)$ with the following property. For every $\mathsf{RCD}(K,N)$ space $(X,\di, \mm)$, the following improved Buser's inequality holds:
\begin{equation}\label{eq:improvedBuser}
h(X)\geq \sup_{t>0}  \frac{1-e^{-\lambda_1t}}{J_K(t)} +\varepsilon.
\end{equation}
\end{corollary}
Let us stress that Corollary \ref{cor:improvedBuser} is new even for smooth Riemannian manifolds with dimension $\leq N$ and Ricci curvature $\geq K>0$.
 \\
 
A second goal of the paper is to study the equality case in Cheeger's inequality \eqref{eq:ChIn}. In a series of now classical papers \cite{Buser78, Buser79,Buser}, Buser proved that equality in Cheeger's inequality is never attained for compact Riemannian manifolds and gave compact examples where the equality is almost attained (up to an error $\varepsilon>0$ arbitrarily small), showing the sharpness of \eqref{eq:ChIn} among smooth manifolds. Since in Buser's examples the diameters of the spaces grow as the error $\varepsilon>0$ decreases, it is natural to ask if Cheeger's inequality can be improved once an upper bound on the diameter is assumed.
Indeed, improvements of Cheeger's inequality when the Ricci curvature  lower  bound is coupled with upper bounds on the dimension and on the diameter have been considered for instance by Gallot in \cite{Ga} (see in particular Section 6) and by Bayle in \cite{Ba} (see equation (2.49) at page 85 and Remark 2.5.3 therein). The improvements are based on the observation that, under these assumptions, the isoperimetric profile has a better than linear behaviour. This can be turned into a better lower bound for the first eigenvalue of the Laplacian with a very general argument (see for instance \cite[Section 6]{Grig1}).\\
Linked to this question, it is also natural to ask if the equality in  \eqref{eq:ChIn} can be attained either in the  \emph{non-compact}  or in  the \emph{non-smooth compact} setting. We prove that the answer is positive for the former and is negative for the latter,  even under more general assumptions. More precisely, the second main result of the paper is the following:

\begin{theorem}\label{thm:rigcheegfinite}
Let $(X,\di,\mm)$ be an $\mathsf{RCD}(K,\infty)$ metric measure space with $\mm(X)<\infty$ and $K\in\R$. Assume that $(X, \di,\mm)$ admits a superlinear isoperimetric profile (this is always satisfied if ${\rm diam}(X)<\infty$ or $K>0$).

Then the equality in Cheeger's inequality is never attained, i.e.
\begin{equation}\label{eq:ChIneStrict}
\lambda_1>\frac{1}{4}h(X)^2.
\end{equation}
\end{theorem}
We refer to the preliminaries given below for the Definition \ref{def:IsopProf} of superlinear isoperimetric profile, and to \autoref{thm:rigcheeginfinite} for the case $\mm(X)=\infty$.

Along the same lines of the arguments for Corollary \ref{cor:improvedBuser}, one can obtain the next improvement of Cheeger's inequality:

\begin{corollary}[Improved Cheeger's inequality]\label{cor:improvedCheeger}
For every $K\in \R$, $N\in [1,\infty)$ and $D\in (0,\infty)$ there exists $\varepsilon=\varepsilon(K,N,D)>0$ with the following property. For every $\mathsf{RCD}(K,N)$ space $(X,\di, \mm)$ with ${\rm diam}(X)\leq D$, the following improved Cheeger's inequality holds:
\begin{equation}\label{eq:improvedCheeger}
\lambda_1\geq \frac{1}{4}h(X)^2 +\varepsilon.
\end{equation}
\end{corollary}

It is a classical fact that Cheeger's inequality fits into a family of inequalities relating eigenvalues of the $p$-Laplacian associated to different exponents $1\le p<\infty$ (see for instance \cite[Theorem 3.2]{Li} for the case of Euclidean domains and Dirichlet boundary conditions or \cite[Proposition 2.5]{Mi1} for general metric measure spaces). In this paper we show that these inequalities are strict for a large class of $\mathsf{RCD}$ metric measure spaces.

\begin{theorem}
Let $(X,\di,\mm)$ be an $\mathsf{RCD}(K,\infty)$ metric measure space with $\mm(X)=1$ and superlinear isoperimetric profile. Then the function
\begin{equation*}
[1,\infty)\ni p\mapsto p\left(\lambda_{1,p}(X)\right)^{\frac{1}{p}}
\end{equation*}
is strictly increasing, where $\lambda_{1,1}(X)=h(X)$ is the Cheeger constant.
\end{theorem}

In the last part of the paper, we provide several examples of spaces with Ricci curvature bounded below where the equality in Cheeger's inequality is attained.
\\

We conclude the introduction by mentioning that rigidity results involving the spectrum of $\mathsf{RCD}$ spaces received a lot of attention in the recent literature, a non-exhaustive list follows: Ketterer \cite{Ket} extended the validity of Obata's rigidity theorem to the non-smooth setting, Cavalletti-Mondino \cite{CaMoGT} proved rigidity results involving the first eigenvalue of the $p$-Laplace operator with Neumann boundary conditions and Mondino-Semola \cite{MoSe} for Dirichlet boundary conditions, Gigli-Ketterer-Kuwada-Ohta \cite{GKKO} established the rigidity in the $\mathsf{RCD}(K,\infty)$ spectral gap, Ambrosio-Bru\'e-Semola \cite{ABS19} proved rigidity in the 1-Bakry-\'Emery inequality of $\mathsf{RCD}(0,N)$ spaces, later extended to $\mathsf{RCD}(K,\infty)$ spaces with $K>0$ by Han \cite{Han}. 

As a final remark, let us also point out the following general principle (clearly presented in the Introduction of \cite{Mi}) which lies behind several of the aforementioned results: the hierarchy between $p$-spectral gaps associated to different exponents $p$ is independent of any curvature assumption in one direction (Cheeger's inequality), while it heavily relies on lower curvature bounds in the other one (Buser's inequality).

\section*{Acknowledgements}
A.M. is supported by the European Research Council (ERC), under the European's Union Horizon 2020 research and innovation programme, via the ERC Starting Grant  “CURVATURE”, grant agreement No. 802689.

D.S. wishes to thank Marco Barberis for a useful conversation about Example \ref{examplethrice}.

The authors would like to thank the anonymous referees for their suggestions that helped to improve a previous version of the paper.

\section{Preliminaries}\label{sec:prel}

\subsection{Curvature bounds and heat flow}\label{subsec:RCD}
Unless otherwise stated, we assume $(X,\di)$ is a complete and separable metric space endowed with a  $\sigma$-finite, non-negative reference measure $\mm$ over the Borel $\sigma$-algebra $\mathcal{B}$. We also assume $\textsf{supp}(\mm)=X$ and the existence of $x_0\in X$, $M>0$ and $c\geq 0$ such that
$$\mm(B_r(x_0))\leq M\exp(cr^2) \ \ \textrm{for\ every} \ r\geq 0\, .$$
Possibly enlarging $\mathcal{B}$ and extending the measure $\mm$, we can assume that $\mathcal{B}$ is $\mm$-complete without loss of generality. 
We call $(X,\mathsf{d},\mm)$ a metric measure space, m.m.s for short.

\vspace{0.3cm}

We denote by $(\mathcal{P}_2(X), W_2)$ the space of probability measures on $X$ with finite second moment endowed with the quadratic Kantorovich-Wasserstein distance $W_{2}$. 

The \textit{relative entropy functional} $\mathsf{Ent}_{\mm}:\mathcal{P}_2(X)\rightarrow \R\cup \{+\infty\}$ is defined as
\begin{equation}
\mathsf{Ent}_{\mm}(\mu):=\begin{cases} \int \rho\log \rho \, \di\mm \ &\textrm{if} \ \mu=\rho\mm\, \ \textrm{and \ $(\rho \log \rho)^{-}\in L^{1}(X,\mm)$} , \\ +\infty \ &\textrm{otherwise}\, .\end{cases}
\end{equation} 
In the sequel we use the notation:
$$
D(\mathsf{Ent}_{\mm}):=\{\mu\in \mathcal{P}_2(X) \,:\, \mathsf{Ent}_{\mm}(\mu)\in \R\}.
$$

\begin{definition}[$\mathsf{CD(K,\infty)}$ condition]
Given $K\in \R$, a m.m.s. $(X,\mathsf{d},\mm)$ verifies the $\mathsf{CD}(K,\infty)$ condition if for any $\mu^0,\mu^1\in D(\mathsf{Ent}_{\mm})$ there exists a $W_2$-geodesic $(\mu_t)$ connecting $\mu^0$ and $\mu^1$ and such that, for any $t\in[0,1]$,
\begin{equation}
\mathsf{Ent}_{\mm}(\mu_t)\leq (1-t)\mathsf{Ent}_{\mm}(\mu_0)+t\mathsf{Ent}_{\mm}(\mu_1)-\frac{K}{2}t(1-t)W_2^2(\mu_0,\mu_1).
\end{equation}
\end{definition}

This class of spaces was introduced independently by Sturm \cite{St} and Lott-Villani \cite{LV}.  

If $(X,\di,\mm)$ satisfies $\mathsf{CD}(K,\infty)$ for $K\in \R$, then, for every   $\alpha,\beta>0$,  the metric measure space $(X,\alpha\di,\beta\mm)$ satisfies the $\mathsf{CD}(K/\alpha^2,\infty)$ condition. In particular, it is not restrictive to assume that a $\mathsf{CD}(K,\infty)$ m.m.s. with $\mm(X)<\infty$ is a probability space. Moreover, $K>0$ implies $\mm(X)<\infty$. 

\vspace{0.3cm}

For $1<p<\infty$, the \textit{$p$-Cheeger energy} is defined as
\begin{equation}\label{eq:defChm}
\mathsf{Ch}_{\mm,p}(f):=\inf\bigg\{\liminf_{n\to \infty}\frac{1}{p}\int_X \mathrm{lip}(f_n)^p\di\mm: f_n\in \mathsf{Lip}_{bs}(X), f_n\rightarrow f \ \mathrm{in} \ L^p(X,\mm)\bigg\}.
\end{equation}
\\As proved in \cite{AGS}, $\mathsf{Ch}_{\mm,p}(f)$ can be represented in terms of the so called \textit{minimal weak upper gradient} $|\nabla f|$ as
$$
\mathsf{Ch}_{\mm,p}(f)=  \frac{1}{p}\int_X |\nabla f|^p\, \di\mm\, .
$$

The $p$-Cheeger energy is a $p$-homogeneous, lower semicontinuous and convex functional on $L^p(X,\mm)$ whose proper domain
$$
W^{1,p}(X,\di,\mm):=\{f\in L^{p}(X,\mm)\,:\, \mathsf{Ch}_{\mm,p}(f)<\infty\}
$$
is a dense linear subspace of $L^p(X,\mm)$. The space $W^{1,p}(X,\di,\mm)$ is Banach when endowed with the norm
$$\| f \|_{W^{1,p}}^p:=\| f\|^p_{L^p}+ p\mathsf{Ch}_{\mm,p}(f).$$
Let us also point out that in \cite{GH} the authors proved that the minimal weak upper gradient of a function $f\in W^{1,p}(X,\di,\mm)\cap W^{1,q}(X,\di,\mm)$, $1<p\leq q<\infty$, is independent of the integrability exponent. We will tacitly rely on this fact in the note.

Let $p\in [1,\infty)$ and let $(X,\di,\mm)$ be a metric measure space with finite measure.
Then we set
\begin{equation}
\lambda_{1,p}(X,\di,\mm):=\inf\left\{\frac{1}{c_p^p(f)}\int_X \mathrm{lip}(f)^p\di\mm, \ : \ f\in \mathsf{Lip}_{bs}, \ f \ \textrm{non $\mm$-a.e. constant} \right\},
\end{equation}
where 
$$c^p_p(f):=\inf_{a\in \mathbb{R}}\int_X |f-a|^p\,\di\mm\, .$$

When the space is clear from the context we will drop the dependence on $(X,\di,\mm)$ by putting $\lambda_{1,p}:=\lambda_{1,p}(X,\di,\mm)$. When $p=2$ we will often use the notation $\lambda_1:=\lambda_{1,2}$.

As pointed out in \cite[Section 9]{AH17} by relying on some results contained in \cite{ADG}, it holds $\lambda_{1,1}(X,\di,\mm)=h(X)$ where $h(X)$ is the Cheeger constant of the space $(X,\di,\mm)$ as defined in \eqref{eq:defChConst}. Moreover, for $1<p<\infty$ we have the equivalent formulation 
\begin{equation}\label{eq: eqdeflambda1p}
\lambda_{1,p}(X,\di,\mm)=\inf\left\{\int_X|\nabla f|^p\di\mm\,:\, f\in\Lambda_{p}(X,\di,\mm)\right\}
\end{equation}
where 
\begin{equation}
\Lambda_p(X,\di,\mm):=\left\{f\in W^{1,p}(X,\di,\mm)\, :\, \int_X|f|^p\di\mm=1\, ,\, \int_X|f|^{p-2}f\di\mm=0\right\}.
\end{equation}

For all $f\in W^{1,2}(X,\di,\mm)$, the subdifferential $\partial\mathsf{Ch}_{\mm,2}(f)$ of the Cheeger energy at $f$ is defined as 
\begin{equation}
\partial\mathsf{Ch}_{\mm,2}(f):=\left\{\ell \in L^2(X,\mm)\ : \ \int_X\ell(g-f)\,\di\mm\leq  \mathsf{Ch}_{\mm,2}(g)-\mathsf{Ch}_{\mm,2}(f) \quad \forall g\in L^2(X,\mm)\right\}.
\end{equation}

We denote by $(H_{t})_{t\geq 0}$, and we refer to it as \textit{heat flow}, the $L^{2}(X,\mm)$-gradient flow of the Cheeger energy, i.e. for any $f\in L^2(X,\mm)$ the map $t\mapsto H_t f$ is a locally Lipschitz map from $(0,\infty)$ to $L^2(X,\mm)$ such that $H_tf\to f$ in $L^2(X,\mm)$ as $t\to 0$ and
\begin{equation}
\diff{}{t} H_{t} f \in -\partial\mathsf{Ch}_{\mm,2}(H_t f) \ \ \textrm{for a.e.} \ t\in (0,\infty)\, .
\end{equation} 

\medskip

We recall now the $\mathsf{RCD}$ condition, a reinforcement of the $\mathsf{CD}$ condition introduced by Ambrosio, Gigli and Savar\'e \cite{AGS1} (in case $\mm(X)<\infty$; see also \cite{AGMR} for the current axiomatization and the extension to $\sigma$-finite measures). 
\begin{definition}[$\mathsf{RCD}(K,\infty)$ condition]
A metric measure space $(X,\mathsf{d},\mm)$ satisfies the $\mathsf{RCD}(K,\infty)$ condition, $K\in \R$, if it is $\mathsf{CD}(K,\infty)$ and the Cheeger energy $\mathsf{Ch}_{\mm,2}$ is quadratic. 
\end{definition}

The quadraticity of the Cheeger energy is equivalent to the fact that $W^{1,2}(X,\di,\mm)$ is Hilbert.  Such an extra requirement singles out the ``Riemannian'' m.m.s structures out of the ``possibly Finsler'' ones.

The set $D(\Delta)$ is defined as the set of $f\in L^2(X,\mm)$ such that $\partial\mathsf{Ch}_{\mm,2}(f)\neq \emptyset$. In particular, $D(\Delta)\subset W^{1,2}(X,\di,\mm)$. For $f\in D(\Delta)$ we define $-\Delta f$ as the element of minimal $L^2(X,\mm)$ norm in $\partial\mathsf{Ch}_{\mm,2}(f)$.

We recall that on an $\mathsf{RCD}(K,\infty)$ space the heat flow can be extended to a linear semigroup of contractions in $L^p(X,\mm)$ for every $p\in [1,\infty)$. For every $f\in L^2(X,\mm)$ and any $t>0$ we have $H_tf\in D(\Delta)$.
The maximum principle ensures that for any $C>0$ and any $f\in L^{2}(X,\mm)$ with $0\le f \le C$ $\mm$-a.e.,  it holds $0\le H_tf\le C$. 
\\ The semigroup $H_t$ admits an $\mm\otimes \mm$-measurable density kernel ${\rm \rho}_{t} (x,y)$, so that
$$H_tf(x)=\int_X f(y){\rm \rho}_{t} (x,y) \, \di\mm(y), \quad  \text{ for } \mm\text{-a.e. } x\in X,  \quad \textrm{for any } f\in L^2(X,\mm).$$
We also know (see \cite[Theorem 6.1]{AGS1}) that, up to a suitable choice of $\mm$-a.e. representative, $H_tf$ belongs to $\mathcal{C}(X)\cap L^{\infty}((0,\infty)\times X)$ whenever $f\in L^{\infty}(X,\mm),$ where $\mathcal{C}(X)$ denotes the set of real valued continuous functions over $X$.
Moreover, for any $f\in L^2\cap L^{\infty}(X,\mm)$ and for every $t>0$ the regularizing property of the heat flow yields $H_tf\in \mathsf{Lip}_{b}(X)$ with the bound (sharp in the case $K>0$) \cite[Proposition 3.1]{DeMo}
\begin{equation}\label{strong Feller}
\begin{aligned}
&\| \,|\nabla H_tf|\, \|_{\infty}\leq \sqrt{\frac{2K}{\pi(e^{2Kt}-1)}} \; \| f\|_{\infty}\, \quad \textrm{if} \ K\neq 0, \\
&\| \,|\nabla H_tf|\, \|_{\infty}\leq \sqrt{\frac{ 1 }{\pi t}}\; \| f\|_{\infty} \quad \textrm{if} \ K=0.
\end{aligned}
\end{equation}

The $1$-Bakry-\'Emery inequality, proved in the $\mathsf{RCD}$ setting by Savar\'e \cite[Corollary 3.5]{Savare}, ensures that
\begin{equation}\label{dis: Savare forte}
|\nabla H_tf|\leq e^{-Kt}H_t(|\nabla f|), \quad \mm\text{-a.e. } \quad \textrm{for any} \ f\in W^{1,2}(X,\di,\mm). 
\end{equation}
\medskip

We next recall the classical notion of \emph{ultracontractivity}.
\begin{definition}\label{def:ultracontractive}
The semigroup $H_t$ is $L^2\to L^{\infty}$ ultracontractive, or simply ultracontractive, if there exists a positive function $\theta(t)$ such that for any $t>0$ and any $f\in L^2(X,\mm)$ we have $H_tf\in L^{\infty}(X,\mm)$ with 
$$\| H_tf \|_{\infty}\leq \theta(t)\| f\|_{2}\, .$$
\end{definition}
Under this assumption, for every $t>0$ and every $x\in X$ there exists a Lipschitz version of the density kernel $\rho_t(x,\cdot)$. Moreover, it is easy to prove (see e.g. \cite[Theorem 14.4]{Grig}) that $H_t$ is ultracontractive if and only if for every $t>0$ and $x\in X$ we have ${\rm \rho}_{2t} (x,x) \leq \theta^2(t)$.

We conclude the section by stating a result which ensures the ultracontractivity of the heat semigroup.

\begin{proposition}\label{prop:Tamaninicond}
Let $(X,\di,\mm)$ be an $\mathsf{RCD}(K,\infty)$ metric measure space for some $K\in \R$.
Assume that there exists a positive function $A(r)$ such that
\begin{equation}\label{eq:VolLB}
\mm(B_r(x))>A(r) \quad \textrm{for every} \ x\in X, \, r\in (0,\infty).
\end{equation}
Then $H_t$ is ultracontractive.
\end{proposition}
\begin{proof}
By a recent result of Tamanini \cite[Corollary 3.3]{Tam}, the heat kernel on an $\mathsf{RCD}(K,\infty)$ space satisfies the following point-wise Gaussian bounds: there exists $C_{K}>0$ depending only on $K$ (if $K\geq 0$, one can choose $C_{K}=0$) and for every $\varepsilon>0$ there exists $C_{\varepsilon}>0$ such that
\begin{equation}\label{eq:heatKernelBound}
0\leq {\rm \rho}_{t} (x,y) \leq \frac{1}{\sqrt{\mm(B_{\sqrt{t}} (x) ) \, \mm(B_{\sqrt{t}} (y) ) }}  \exp\left(C_{\varepsilon} (1+C_{K} t)  - \frac{\di(x,y)^{2}}{(4+\varepsilon)t} \right).
\end{equation}
The assumption  \eqref{eq:VolLB} combined with \eqref{eq:heatKernelBound} gives that the heat kernel is uniformly bounded from above on the diagonal.
This implies the ultracontractivity.
\end{proof}

\subsection{Functions of bounded variation and perimeter}\label{subsec:BV}

\begin{definition}[${\rm BV}$ space]
A function $f\in L^1(X,\mm)$ belongs to the space ${\rm BV}(X,\di,\mm)$ of functions of bounded variation (see \cite{Mir,AD}) if there exists a sequence $(f_n)_{n\in\mathbb{N}}\in \mathsf{Lip}_{loc}(X)$ converging to $f$ in $L^1(X,\mm)$ and such that
$$\limsup_{n\to \infty}\int_X \mathrm{lip}(f_n)\, \di\mm <+\infty\, .$$

If $f\in {\rm BV}(X,\di,\mm)$ and $A\subset X$ is open, we define 
\begin{equation}\label{eq:defTV}
|Df|(A):=\inf\bigg\{\liminf_{n\rightarrow \infty}\int_A \mathrm{lip}(f_n)\,\di\mm: f_n\in \mathsf{Lip}_{loc}(X), f_n\rightarrow f \ \mathrm{in} \ L^1(A,\mm)\bigg\}.
\end{equation}
\end{definition}

It is known that this function is the restriction to open sets of a finite Borel measure, called \textit{total variation} of $f$ and denoted by $|Df|$.
\\By the very definition of $|Df|(X)$, it is immediate to see that for all $f, f_n\in {\rm BV}(X,\di,\mm)$
\begin{equation}\label{eq:lscbv}
|Df|(X)\leq \liminf_n |Df_n|(X) \quad \textrm{whenever} \ f_n\rightarrow f \ \textrm{in} \ L^1(X,\mm).
\end{equation}
Moreover, for all $\varphi:\R\to \R$ $1$-Lipschitz with $\varphi(0)=0$ we have 
$$|D(\varphi\circ f)|(X)\leq |Df|(X).$$

When $(X,\di,\mm)$ is an  $\mathsf{RCD}(K,\infty)$ space and $f\in {\rm BV}(X,\di,\mm)$, a result of Ambrosio  and Honda \cite[Proposition 1.6.3]{AH17} ensures that $H_tf \in {\rm BV}(X,\di,\mm)$ with the explicit inequality
\begin{equation}\label{dis: 1-BE for BV}
|DH_tf|(X)\leq e^{-Kt}|Df|(X)\, .
\end{equation}

Given a Borel set $E$ of finite measure, we say that $E$ is a set of finite perimeter if $\chi_E\in {\rm BV}(X,\di,\mm)$ and we set $\Per(E):=|D\chi_E|(X)$, i.e.
\begin{equation}\label{def:per}
\mathrm{Per}(E)=\inf\bigg\{\liminf_{n\rightarrow \infty}\int_X \mathrm{lip}(f_n)\,\di\mm: f_n\in \mathsf{Lip}_{loc}(X), f_n\rightarrow \chi_E \ \mathrm{in} \ L^1(X,\mm)\bigg\}.
\end{equation}
We remark that we can replace the set $\mathsf{Lip}_{loc}(X)$ in definition \eqref{def:per} with the set $\mathsf{Lip}_{bs}(X)$, and we can also suppose that $0\le f_n \le 1$ (see \cite[Remark 3.4, 3.5]{ADG}).

\begin{proposition}[Coarea inequality and coarea formula]\label{prop:Coarea}
Let $(X,\di)$ be a complete metric space and let $\mm$ be a non-negative Borel measure finite on bounded subsets.
\\Let $f\in {\mathsf {Lip}}_{bs}(X)$, $f:X\to [0,\infty)$ and set $M=\sup_{X} f$. 
Then for $\mathcal{L}^{1}$-a.e. $t>0$ the set $\{f> t \}$ has finite perimeter and 
\begin{equation}\label{eq:coarea}
\int_{0}^{M} \Per(\{f>t\}) \, \di t \leq \int_{X} |{\rm lip}(f)| \, \di \mm.
\end{equation}

If in addition $(X,\di)$ is  separable, then the coarea formula for $\rm{BV}$ functions holds, i.e. for every $f:X\to [0,\infty)$ with  $f\in {\rm BV}(X, \di,\mm)$ it holds
\begin{equation}\label{eq:coareaBV}
\int_{0}^{\infty} \Per(\{f>t\}) \, \di t= |Df|(X).
\end{equation}
\end{proposition}

\begin{proof}
For a proof of the first part, see for instance  \cite[Proposition 3.5]{DeMo}. 
\\The second claim was already observed in the introduction of \cite{AD} and can be proved along the lines of \cite{Mir}. 
\end{proof}

The next corollary will be useful later in the paper.
\begin{corollary}[Finiteness of the Cheeger constant]\label{cor:h(X)finite}
Let $(X,\di,\mm)$ be as in Proposition \ref{prop:Coarea} (first part)  with  $\mm(X)\in (0,\infty]$ and ${\rm diam}(X)> 0$. Then  the Cheeger constant  defined in \eqref{eq:defChConst} is finite, i.e. $h(X)\in [0,\infty)$.
\end{corollary}
\begin{proof}
Let $x_0\in {\rm supp} \,\mm$ and consider $f(\cdot):=\max\{1-\di(x_0, \cdot), 0\}\in \mathsf{Lip}_{bs}(X)$.  By the non triviality assumptions on $(X,\di,\mm)$ and the coarea inequality (Proposition \ref{prop:Coarea}), it follows that there exists $r\in (0,1)$ such that the metric ball $B_r(x_0)$ satisfies $\mm(B_r(x_0))\in (0,\mm(X)/2)$ and $\Per(B_r(x_0))\in (0,\infty)$. Thus the set of competitors with finite energy in the variational problem \eqref{eq:defChConst} defining $h(X)$ is non empty and the conclusion follows.
\end{proof}
\vspace{0.3cm}

Let $K>0$, we denote by $I_K:[0,1]\to [0,\sqrt{K/(2\pi)}]$ the Gaussian isoperimetric profile function defined as $I_K:=\varphi_K\circ \Phi_K^{-1}$, where
$$\Phi_K(x):=\sqrt{\frac{K}{2\pi}}\int_{-\infty}^x e^{-Kt^2/2}\, \di t, \ \ x\in \mathbb{R},$$
and $\varphi_K:=\Phi_K'$. The function $I_K$ satisfies $I_K(1/2)=\sqrt{K/(2\pi)}$ and $I_K(x)=\sqrt{K}I_1(x)$. For simplicity of notation we set $I:=I_{1}$. We recall the following asymptotic (see \cite{BL})
\begin{equation}\label{id: asintotico I}
\lim_{x\to 0} \frac{I_{K}(x)}{x\sqrt{2K\log{\frac{1}{x}}}}=1 .
\end{equation}

As proved by Ambrosio and Mondino \cite[Theorem 4.2]{AM16}, the celebrated Gaussian isoperimetric inequality of Bakry-Ledoux \cite{BL} extends to the class of $\mathsf{RCD}(K,\infty)$ spaces with positive $K$:
\begin{proposition}\label{prop:Gauss isop}
Let $(X,\di,\mm)$ be an $\mathsf{RCD}(K,\infty)$ space with $\mm(X)=1$ and $K>0$. Then, for every Borel subset $A\subset X$ we have
\begin{equation}\label{eq:GII}
\Per(A)\geq I_K(\mm(A)).
\end{equation}
\end{proposition}

Thanks to a recent result by Han \cite[Corollary 4.4]{Han} building on top of \cite{ABS19}, we also have at our disposal a rigidity statement for the Gaussian isoperimetric inequality. 

\begin{proposition}\label{prop: Han rigidity}
Let $(X,\mathsf{d},\mm)$ be an $\mathsf{RCD}(K,\infty)$ metric measure space with $\mm(X)=1$ and $K>0$. Let us suppose there exists a Borel set $E\subset X$ with positive measure such that 
\begin{equation}\label{eq: equality isoperimetric Han}
\Per(E)=I_K(\mm(E)).
\end{equation}
Then 
$$(X,\di,\mm)\cong (Y,\di_Y,\mm_Y)\times (\R,|\cdot|,\sqrt{K/(2\pi)}e^{-Kt^2/2}\di t)$$
for some $\mathsf{RCD}(K,\infty)$ space $(Y,\di_Y,\mm_Y)$, where $\cong$ denotes  isomorphism as metric measure spaces. 
\end{proposition}

We will take for granted the following key lemma, which can be obtained as a simpler variant of \cite[Lemma 1.5.8]{AH17} in the case of a fixed ambient space.

\begin{lemma}\label{lemma:weak}
Let $(X,\di,\mm)$ be an $\mathsf{RCD}(K,\infty)$ metric measure space. Let $(f_{k})_{k\in \mathbb{N}}$ be a sequence converging in $L^2(X,\mm)$ to $f$ and assume that
$$\sup_{k\in \mathbb{N}} \mathsf{Ch}_{\mm,2}(f_k)<\infty.$$
Then, for any lower semicontinuous function $g:X\to[0,\infty]$, it holds that
\begin{equation}\label{eq:techlsc}
\int_Xg|\nabla f|\di\mm\le\liminf_{k\to\infty}\int_Xg|\nabla f_{k}|\di\mm.
\end{equation} 	
\end{lemma}

Another key property is the compactness in ${\rm BV}$. In order to state the result, we recall the next crucial definition.
\begin{definition}\label{def:IsopProf}
Let $(X,\di, \mm)$ be a m.m.s. with $\mm(X)=1$.  We shall say that $(X,\di,\mm)$ admits a superlinear isoperimetric profile if there exists a function $\omega:(0,\infty)\to (0,1/2]$ such that for all $\varepsilon>0$ it holds: 
\begin{equation}\label{eq:ExIsopProf}
\mm(E) \leq \omega(\varepsilon) \quad \Longrightarrow \quad \mm(E) \leq \varepsilon \, \mathrm{Per}(E)
\end{equation}
for any Borel set $E\subset X$.
\end{definition}

Notice that, classically, a metric measure space $(X,\di,\mm)$ such that $\mm(X)=1$ is said to admit an isoperimetric profile $I:[0,1]\to(0,\infty)$ if for any Borel set $E\subset X$ it holds that
\begin{equation*}
\Per(E)\ge I(\mm(E))\, . 
\end{equation*}

It is easy to check that the Gaussian isoperimetric inequality (Proposition \ref{prop:Gauss isop}) combined with the asymptotic \eqref{id: asintotico I} imply that if $(X,\di, \mm)$ is an $\mathsf{RCD}(K,\infty)$ space for some $K>0$, then it admits a  superlinear  isoperimetric profile.
We also know \cite[Theorem 7.2]{AH17} that any $\mathsf{RCD}(K,\infty)$ space with finite diameter has a superlinear isoperimetric profile. 

\begin{proposition}[Compactness in ${\rm BV}$ and $L^2$]\label{prop:compBV}
Let $(X,\di,\mm)$ be an $\mathsf{RCD}(K,\infty)$ metric measure space with $\mm(X)<\infty$ admitting a superlinear  isoperimetric profile (this is satisfied in case $K>0$ or ${\rm diam}(X)<\infty$). Then, for any sequence of functions $(f_{k})_{k\in \mathbb{N}}\subset \rm{ BV}(X,\di,\mm)$ such that 
\begin{equation}
\sup_k\left\{\int_X|f_{k}|\di\mm+|Df_{k}|(X)\right\}<\infty 
\end{equation}	
there exist  a function $f\in\rm{BV}(X,\di,\mm)$ and a subsequence $k_{j}$ such that 
\begin{equation}
{\rm sign}(f_{k_j})\sqrt{|f_{k_j}|} \rightarrow {\rm sign}(f)\sqrt{|f|} \qquad \textrm{in} \ L^2(X,\mm).
\end{equation}
Moreover, for any $1<p<\infty$ and for any sequence of functions $(f_{k})_{k\in \mathbb{N}}\subset W^{1,p}(X,\di,\mm)$ such that 
\begin{equation}
\sup_k \| f_k \|_{W^{1,p}}<\infty 
\end{equation}	
there exist  a function $f\in L^p(X,\mm)$ and a subsequence $k_{j}$ such that 
\begin{equation}
f_{k_j} \rightarrow f \qquad \textrm{in} \ L^p(X,\mm).
\end{equation}
\end{proposition}

\begin{proof}
The argument can be obtained arguing as in the proof of \cite[Proposition 1.7.5]{AH17}.
\end{proof}

\begin{proposition}[Stability in ${\rm BV}$]\label{prop:stabBV}
Let $(X,\di,\mm)$ be an $\mathsf{RCD}(K,\infty)$ metric measure space. Let $(f_{k})_{k\in \mathbb{N}}\subset \rm{ BV}(X,\di,\mm)$  be such that ${\rm sign}(f_k)\sqrt{|f_k|}$ converge to ${\rm sign}(f)\sqrt{|f|}$ in $L^2(X,\mm)$ and assume that $\sup_k|Df_k|(X)<\infty$. Then 	$f\in {\rm BV}(X,\di,\mm)$ and 
\begin{equation*}
|Df|(X)\le\liminf_{k\to\infty}|Df_k|(X).
\end{equation*}
\end{proposition}

\begin{proof}
The proof is strongly inspired by \cite[Theorem 1.6.4]{AH17}, we report it for reader's convenience.

\textbf{Step 1.} In the first step we reduce to the case of uniformly bounded functions. To this aim we observe that, for any $N\in\mathbb{N}$ the truncated functions $f_k^N:=N\wedge f_k\vee -N$ and $f^N:=N\wedge f\vee -N$ verify the assumptions of the statement. Moreover, for any $g\in L^1(X,\mm)$ it holds that $g^N\to g$ in $L^1$ as $N\to\infty$.
Then, if we are able to prove that
\begin{equation*}
|Df^N|(X)\le\liminf_{k\to\infty}|Df^N_k|(X),
\end{equation*}	
the general conclusion will follow by lower semicontinuity of the variation with respect to $L^1$ convergence (see \eqref{eq:lscbv}) recalling that $|Dg^N|(X)\le |Dg|(X)$ for any $N\in\mathbb{N}$ and any $g\in {\rm BV}(X,\di,\mm)$. 

\textbf{Step 2.} 
Let us fix now $t>0$ and observe that the functions $H_tf_k$ are uniformly bounded, uniformly Lipschitz, they belong to $W^{1,2}(X,\di,\mm)$ and they still verify the assumptions of the statement. If we are able to prove that
\begin{equation*}
|DH_tf|(X)\le\liminf_{k\to\infty}|DH_tf_k|(X),
\end{equation*}	
then the conclusion will follow by the $1$-Bakry-\'Emery inequality, yielding $|DH_tf_k|(X)\le e^{-Kt}|Df_k|(X)$, and the lower semicontinuity of the total variation w.r.t. $L^1$ convergence again, passing to the $\liminf$ as $t\downarrow 0$.

Recalling the representation formula for the total variation \cite[Proposition 1.6.3 (a)]{AH17}  
\begin{equation*}
|Df|(X)=\int_X|\nabla f|\di\mm
\end{equation*}
for any $f\in{\rm Lip}_{\rm b}(X)\cap L^1(X,\mm)\cap W^{1,2}(X,\di,\mm)$, it remains to prove that
\begin{equation}\label{eq:int2}
\int_X|\nabla f|\di\mm\le\liminf_{k\to\infty}\int_X|\nabla f_k|\di\mm,
\end{equation} 
whenever $(f_k)\subset W^{1,2}(X,\di,\mm)$ are uniformly bounded in $W^{1,2}(X,\di,\mm)$ and converge to $f$ in $L^2(X,\mm)$.

\textbf{Step 3.} To conclude we just point out that \eqref{eq:int2} above follows from Lemma \ref{lemma:weak} choosing $g\equiv 1$.
\end{proof}

Let us state and prove a general existence result for optimizers of the variational problem defining the Cheeger constant \eqref{eq:defChConst}.

\begin{proposition}\label{prop:existenceofCheegersets}
Let $(X,\di,\mm)$ be an $\mathsf{RCD}(K,\infty)$ metric measure space with $\mm(X)=1$ and admitting a superlinear isoperimetric profile (this is satisfied in case $K>0$ or ${\rm diam}(X)<\infty$). Then there exists a Borel set $E\subset X$ with finite perimeter and $0<\mm(E)\le 1/2$ which is an optimizer for the variational problem defining the Cheeger constant \eqref{eq:defChConst}, i.e.
\begin{equation}
\frac{\Per(E)}{\mm(E)}=h(X).
\end{equation}	
\end{proposition}

\begin{proof}
First of all we note that $h(X)\in [0,\infty)$ in virtue of Corollary \ref{cor:h(X)finite}. 
By the very definition of the Cheeger constant $h(X)$ we can find a sequence of Borel sets of finite perimeter $E_n\subset X$ such that $\mm(E_n)\le 1/2$ for any $n\in\mathbb{N}$ and 
\begin{equation}\label{eq:chopt}
\lim_{n\to\infty}\frac{\Per(E_n)}{\mm(E_n)}=h(X).
\end{equation}	

Let us set $f_n:=\chi_{E_n}$, where we denote by $\chi_F:X\to\{0,1\}$ the indicator function of an arbitrary Borel set $F\subset X$. We claim that $\mm(E_n)$ are bounded away from $0$. If this is not the case, i.e. $\liminf_{n\to\infty}\mm(E_n)=0$, from the existence of a superlinear isoperimetric profile \eqref{eq:ExIsopProf} for $(X, \di, \mm)$ we infer that 
\begin{equation*}
\liminf_{n\to\infty}\frac{\Per(E_n)}{\mm(E_n)}=\infty\, ,
\end{equation*}
contradicting \eqref{eq:chopt}.

It follows from \eqref{eq:chopt} that the functions $f_n$ have uniformly bounded $\mathrm{BV}$-norms, since $|Df_n|(X)=\Per(E_n)$ and $||f_n||_{L^1}=\mm(E_n)$. Thus by Proposition \ref{prop:compBV} there exist a function $f\in \mathrm{BV}(X,\di,\mm)$ and a subsequence, that we do not relabel, such that ${\rm sign}(f_n)\sqrt{|f_n|}\to {\rm sign}(f)\sqrt{|f|}$ in $L^2(X,\mm)$. In particular, 
\begin{equation}\label{eq:cnorms}
||f||_{L^1}=\lim_{n\to\infty}\mm(E_n).
\end{equation}
We claim that there exists a Borel set $E\subset X$ such that $f=\chi_E$,  $\mm$-a.e.. In order to prove the claim it is sufficient to verify that 
\begin{equation}\label{eq:f1-f=0}
f(1-f)=0, \quad \mm\text{-a.e.} . 
\end{equation}
To this aim we observe that $f_n(1-f_n)=0$, $\mm$-a.e. for any $n\in\mathbb{N}$, since $f_n$ are indicator functions. Moreover, using that $f_n\to f$ in $L^1(X,\mm)$ and $\|f_n\|_{L^\infty(X,\mm)}=1$, it is easily seen that $f_n(1-f_n)\to f(1-f)$ in $L^1(X,\mm)$,  proving the claim  \eqref{eq:f1-f=0}. 

Combining \eqref{eq:cnorms} with  \eqref{eq:f1-f=0}  we obtain that
\begin{equation}\label{eq:cmeas}
\mm(E)=\lim_{n\to\infty}\mm(E_n).
\end{equation}
In particular $0<\mm(E)\le 1/2$.
\\By Proposition \ref{prop:stabBV} we also infer that  
\begin{equation}\label{eq:scper}
\Per(E)\le\liminf_{n\to\infty}\Per(E_n).
\end{equation} 
The combination of \eqref{eq:chopt}, \eqref{eq:cmeas} and \eqref{eq:scper} yields that
\begin{equation*}
h(X)\le \frac{\Per(E)}{\mm(E)}\le \liminf_{n\to\infty}\frac{\Per(E_n)}{\mm(E_n)}=h(X)
\end{equation*}
 and we obtain that $E$ is an optimizer for the variational problem defining the Cheeger constant.
\end{proof}

\begin{remark}
Proposition \ref{prop:existenceofCheegersets} should be compared with \cite[Lemma 3.4]{Buser} where the author shows that for any compact Riemannian manifold $M$ there exist a $v>0$ and a sequence of open submanifolds $V_k\subset M$ such that $\mathrm{vol}(V_k)=v$ and $h(M)=\lim_{k\rightarrow \infty} \Per(V_k)/\mathrm{vol}(V_k)$. This is used to apply the deep regularity theory for minimizing currents and then argue via a Heintze-Karcher type estimate. One of the advantages of the present approach is that it does not rely on the regularity theory for minimizing currents.
\end{remark}

The next result will be useful later in the proof of \autoref{thm:rigcheeginfinite}.
\begin{lemma}\label{lem:fW12f2BV}
Let $(X,\di,\mm)$ be a complete and separable metric measure space with $\mm$ finite on bounded sets, and let $f\in W^{1,2}(X,\di,\mm)$.
Then $ f^{2}\in {\rm BV}(X,\di,\mm)$ and
\begin{equation}\label{eq:f2BV}
|D(f^2)|\leq 2\, |f| \, |\nabla f|\, \mm \quad \text{as measures.}
\end{equation}
 \end{lemma}
 
 \begin{proof}
For brevity, we refer to \cite{AGS,AD} for the notions of $p$-test plans and $p$-a.e. curve, $p=1,2$.
From \cite[Proposition 5.7]{AGS}, if $f\in W^{1,2}(X,\di,\mm)$ then $f$ is Sobolev along $2$-a.e. curve and
$$
\left| \frac{\di}{\di t} f\circ \gamma \right| \leq  |\nabla f| \circ \gamma \; |\dot \gamma|, \; \text{a.e. in }[0,1], \; 2\text{-a.e. } \gamma\in {\rm AC}((0,1); (X,\di)). 
$$ 
At this point, it is easy to check that $f^{2}$ is a ${\rm BV}$ function in the ``weak'' formulation of \cite[Definition 5.5]{AD} with ``weak'' total variation satisfying $|D(f^{2})|_{w}\leq 2 |f|\,  |\nabla f| \, \mm$ as measures. 
\\We conclude thanks to the identification result \cite[Theorem 1.1]{AD}. 
 \end{proof}

\subsection{Spectrum of the Laplacian}\label{subsec:spectrum}
We start by recalling some classical notions of spectral theory (see for instance \cite[Appendix A]{BGL}).

Let $(X,\di,\mm)$ be an $\mathsf{RCD}(K,\infty)$ metric measure space for some $K\in \R$ and denote by $\Delta$ the Laplacian defined in Section \ref{subsec:RCD}. We know that $-\Delta$ is a densely defined, self-adjoint operator on the Hilbert space $L^2(X,\mm).$
\\A number $\lambda\in \mathbb{C}$ is a \emph{regular value} of $-\Delta$ if $(\lambda \rm{Id}+\Delta)$ has a bounded inverse. The \emph{resolvent set} $\rho(-\Delta)$ is the set of regular values of $-\Delta$, while the \emph{spectrum} $\sigma(-\Delta)$ is defined as $\sigma(-\Delta):=\mathbb{C}\setminus \rho(-\Delta)$. Since $-\Delta$ is nonnegative, it holds $\sigma(-\Delta)\subset [0,\infty)$.
\\ If there exists a non-zero function $f\in D(\Delta)$ such that $-\Delta f=\lambda f$, we call $\lambda\in \sigma(-\Delta)$ an \emph{eigenvalue} and $f$ the associated \emph{eigenfunction}. The set of all eigenvalues forms the so-called \emph{point spectrum}.
\\ The \emph{discrete spectrum} $\sigma_d(-\Delta)$ is the set of all eigenvalues that are isolated in the point spectrum with the corresponding eigenspace which is  finite dimensional. The \emph{essential spectrum} is the closed set defined as $\sigma_{ess}(-\Delta):=\sigma(-\Delta)\setminus \sigma_{d}(-\Delta)$. 
We denote by $\Sigma$ the infimum of the essential spectrum of $-\Delta$, i.e.
$$\Sigma:=\inf \sigma_{ess}(-\Delta) \ \ \textrm{and}\ \ \ \Sigma:=+\infty \ \textrm{if} \ \sigma_{ess}(-\Delta)=\emptyset.$$

The space $W^{1,p}(X,\di,\mm)$ is \emph{compactly embedded} in $L^p(X,\mm)$, $1<p<\infty$, if for any sequence $\{f_k\}\subset W^{1,p}(X,\di,\mm)$ with uniformly bounded $W^{1,p}$-norm we can extract a subsequence $f_{k_j}$ strongly converging in $L^p(X,\mm)$.

In the next theorem, we collect some results about the spectrum of $\mathsf{RCD}(K,\infty)$ spaces which will be useful later on. For the compact embedding $W^{1,2}(X,\di,\mm) \subset \subset L^2(X,\mm)$ in $\mathsf{RCD}(K,\infty)$ spaces under different assumptions see  \cite[Proposition 6.7]{GMS}.
 
\begin{theorem}\label{thm:compembedding}
Let $(X,\di,\mm)$ be an $\mathsf{RCD}(K,\infty)$ metric measure space for some $K\in\R$. 

\vspace{0.3cm}

\textbf{Case $\mm(X)<\infty$}. Assume that $(X,\di,\mm)$  admits a superlinear isoperimetric profile and that $\mm(X)<\infty$. 
Then:
\begin{enumerate}
\item  The embedding of $W^{1,p}(X,\di,\mm)$ in $L^p(X,\mm)$ is compact for any $1<p<\infty$. 
\item The heat semigroup $H_{t}:L^{2}(X,\mm) \to L^{2}(X,\mm)$ is a compact operator, for any $t>0$.
\item  The spectrum of $-\Delta$ is discrete, non-negative and it diverges to $+\infty$, i.e. $\sigma(-\Delta)=\sigma_d(-\Delta)=(\lambda_{k})_{k\in \mathbb N \cup \{0\}} \subset [0,\infty)$ with $\lambda_{k}\to \infty$ as $k\to \infty$.   Moreover:
\begin{equation}\label{eq:lambda1Variational}
\lambda_{0}=0<\lambda_{1}=\min\left \{ 2{\mathsf {Ch}}_{\mm}(f) :  f\in L^{2}(X,\mm), \, \|f\|_{2}=1,\, \int_{X} f\, \di \mm=0 \right\}.
\end{equation}
In particular, $\lambda_1$ coincides with the value introduced in \eqref{eq:defla1Intro}.
 
Moreover, $f\in W^{1,2}(X,\di,\mm)$ is a $\lambda_{0}$-eigenfunction if and only if $f$ is a non-zero constant function, while $f\in W^{1,2}(X,\di,\mm)$ is a (normalized) $\lambda_{1}$-eigenfunction if and only if
\begin{equation}\label{eq:eig1}
\int_X f^2 \, \di\mm=1,\,\,\,\int_X f \, \di\mm=0\,\,\text{and}\,\,\int_X|\nabla f|^2 \, \di\mm=\lambda_1.
\end{equation}
\end{enumerate}

\vspace{0.3cm}

\textbf{Case $\mm(X)=\infty$}. Assume $\mm(X)=\infty$ and  $\lambda_0:=\inf{\sigma_d(-\Delta)}<\Sigma$.
Then:
\begin{enumerate}
\item The eigenvalue $\lambda_0$ satisfies
$$0<\lambda_{0}=\min\{ 2{\mathsf {Ch}}_{\mm}(f) :  f\in L^{2}(X,\mm), \, \|f\|_{2}=1 \}.$$ 
In particular, $\lambda_0$ coincides with the value introduced in \eqref{eq:defla0Intro}.

Moreover, $f\in W^{1,2}(X,\di,\mm)$ is a (normalized) $\lambda_{0}$-eigenfunction if and only if
\begin{equation}\label{eq:eig0}
\int_X f^2 \, \di\mm=1\,\,\text{and}\,\,\int_X|\nabla f|^2 \, \di\mm=\lambda_0.
\end{equation}
\end{enumerate} 
\end{theorem}

\begin{proof}

\textbf{Case $\mm(X)<\infty$.}
Since $X$ admits a superlinear isoperimetric profile, the fact that the embedding of $W^{1,p}(X,\di,\mm)$ in $L^p(X,\mm)$ is compact is a direct consequence of Proposition \ref{prop:compBV}. 
We can thus appeal to the standard spectral theory (see e.g. \cite[Theorem A.6.4]{BGL}) to infer that $H_t$ is a compact operator for every $t>0$ and that $\sigma(-\Delta)$ consists of a sequence of isolated eigenvalues $(\lambda_{k})_{k\in \mathbb N}\subset [0,\infty)$ with $\lambda_{k}\to \infty$ as $k\to \infty$ and finite dimensional associated eigenspaces. 

We can also appeal to the variational characterisation of the eigenvalues (see \cite[Theorem 4.5.1]{Dav} to infer that
\begin{equation}\label{eq:VarCharEV}
\lambda_{k}=\min_{S_{k+1}} \,\max_{f\in S_{k+1}, \|f\|_{2}=1} {2\mathsf{Ch}}_{\mm}(f),
\end{equation} 
where $S_{k}\subset W^{1,2}(X,\di, \mm)$ denotes an arbitrary $k$-dimensional subspace.

Since $f\equiv 1$ is an element of $L^{2}(X,\mm)$, we infer that $\lambda_0=0$. Moreover, from the Sobolev-to-Lipschitz property satisfied by  $\mathsf{RCD}(K,\infty)$ spaces \cite{AGS1}, it holds that $\mathsf{Ch}_{\mm,2}(f)=0$ if and only if $f$ is constant $\mm$-a.e.. 
Thus $\lambda_{0}=0$ if and only if $\mm(X)<\infty$ and,  in this case, $f$ is a $\lambda_{0}$-eigenfunction if and only if $f$ is constant $\mm$-a.e..

Specialising  \eqref{eq:VarCharEV} to $k=1$ in case $\mm(X)<\infty$ gives the claimed  variational formula for $\lambda_{1}$ which trivially coincides with \eqref{eq:defla1Intro} by the very definition of Cheeger energy.

We finally prove the implication ``if $f$ satisfies  \eqref{eq:eig1} then $f$ is a $\lambda_{1}$-eigenfunction''.    
We first show that $\lambda_1 f\in \partial\mathsf{Ch}_{\mm,2}(f)$:  for every $g\in L^2(X,\mm)$ with $\int_{X} g \, \di \mm=0$ we have
$$\mathsf{Ch}_{\mm,2}(g)\geq \frac{\lambda_1}{2}\int_X g^2\,\di\mm\geq \lambda_1\int_X fg\,\di\mm-\frac{\lambda_1}{2}\int_X f^2\di\mm=\mathsf{Ch}_{\mm,2}(f)+\int_X \lambda_1f(g-f)\di\mm $$
as a consequence of the variational characterization of $\lambda_{1}$, a trivial inequality and \eqref{eq:eig1}.
The function $\lambda_1f$ is also the element of minimal norm in the set $\partial\mathsf{Ch}_{\mm,2}(f)$, and thus $-\Delta f=\lambda_1f$, since 
$$\| \Delta f\|_{2}\ge \int_X -f\Delta f\,\di\mm=\int_X |\nabla f|^2\,\di\mm=\lambda_1=\| \lambda_1 f\|_{L^2}\, ,$$
where we have used again \eqref{eq:eig1} and the Cauchy-Schwarz inequality.

\vspace{0.3cm}
\textbf{Case $\mm(X)=\infty$.}
The assumption on the spectrum implies the existence of the first eigenvalue $\lambda_0$ of $-\Delta$ which is isolated and with non-empty finite dimensional eigenspace. Using the Sobolev-to-Lipschitz property satisfied by  $\mathsf{RCD}(K,\infty)$ spaces \cite{AGS1} and observing that the constant functions are not in $L^2(X,\mm)$, we have $\lambda_0>0$. 
The variational characterization of $\lambda_{0}$ still holds (see in this case \cite[Theorem 4.5.2]{Dav}) and gives
$$\lambda_0=\min_{\|f\|_{2}=1}\Big\{2\mathsf{Ch}_{\mm,2}(f)\Big\},$$
so that $\lambda_0$ trivially corresponds to the definition given in \eqref{eq:defla0Intro} by the very definition of Cheeger energy. It remains to prove that a function $f$ such that
$$\int_X f^2\di\mm=1 \ \ \textrm{and} \ \ \int_X |\nabla f|^2\di\mm=\lambda_0$$
is a $\lambda_0$-eigenfunction. 
We first show that $\lambda_0 f\in \partial\mathsf{Ch}_{\mm,2}(f)$:  for every $g\in L^2(X,\mm)$ we have
$$\mathsf{Ch}_{\mm,2}(g)\geq \frac{\lambda_0}{2}\int_X g^2\,\di\mm\geq \lambda_0\int_X fg\,\di\mm-\frac{\lambda_0}{2}\int_X f^2\di\mm=\mathsf{Ch}_{\mm,2}(f)+\int_X \lambda_0f(g-f)\di\mm $$
as a consequence of the variational characterization of $\lambda_{0}$, a trivial inequality and \eqref{eq:eig0}.
The function $\lambda_0f$ is also the element of minimal norm in the set $\partial\mathsf{Ch}_{\mm,2}(f)$, and thus $-\Delta f=\lambda_0f$, since 
$$\| \Delta f\|_{2}\ge \int_X -f\Delta f\,\di\mm=\int_X |\nabla f|^2\,\di\mm=\lambda_0=\| \lambda_0 f\|_{L^2}\, ,$$
where we have used again \eqref{eq:eig1} and the Cauchy-Schwarz inequality.
\end{proof}

\section{Proof of \autoref{thm:Main}}               
\begin{proof}[Proof of \autoref{thm:Main}]
By the scaling property of the $\mathsf{RCD}(K,\infty)$ condition, it is enough to show the result in the case $K=1$.  

Along the proof of \autoref{thm:MainDeMo} in \cite{DeMo} (see in particular \cite[Eq. (49), (50), (51)]{DeMo}), it is shown that for all $t>0$ and for all $A\subset X$ Borel it holds
\begin{multline}\label{in: last part DeMo}
J_1(t)\mathrm{Per}(A)\geq 2\Big(\mm(A)-\mm(A)^2-\left\Vert H_{t/2}(\chi_A-\mm(A))\right\Vert^2_{2}\Big)\\
\geq 2\Big(\mm(A)-\mm(A)^2-e^{-\lambda_1 t}\left\Vert \chi_A-\mm(A)\right\Vert^2_{2}\Big)
=2\mm(A)(1-\mm(A))(1-e^{-\lambda_1t}).
\end{multline}
The inequality \eqref{implicitlambda1Intro} then follows directly from \eqref{in: last part DeMo} by minimizing over all the Borel subsets $A$ with $\mm(A)\leq 1/2$.

By Proposition \ref{prop:existenceofCheegersets} and assumption \eqref{eq: equality Buser} there exists a Borel set $E\subset X$ with $0<\mm(E)\le 1/2$  such that
\begin{equation}\label{eq: Buser con cheeger raggiunta}
\frac{\Per(E)}{\mm(E)}=\sup_{t>0}\frac{1-e^{-\lambda_1t}}{J_1(t)}\, .
\end{equation}

Since $(1-e^{-\lambda_1t})/J_1(t)$ is a continuous function of $t\in (0,\infty)$, we have to analyse three different cases:
\begin{enumerate}
\item the supremum in the right hand side of \eqref{eq: Buser con cheeger raggiunta} is achieved as $t\to 0$;
\item the supremum in the right hand side of \eqref{eq: Buser con cheeger raggiunta} is achieved as $t\to \infty$;
\item the supremum in the right hand side of \eqref{eq: Buser con cheeger raggiunta} is achieved for a certain $\bar{t}\in (0,\infty).$
\end{enumerate}

\textbf{Case} (1): this case is easily ruled out since $(1-e^{-\lambda_1t})/J_1(t)$ is a positive function in $(0,\infty)$ and
$$\lim_{t \to 0} \frac{1-e^{-\lambda_1t}}{J_1(t)}=0.$$
In particular, the supremum can never occur as $t\to 0$.

\textbf{Case} (2): in this situation we have
$$\sup_{t>0}\frac{1-e^{-\lambda_1t}}{J_1(t)}=\frac{1}{\sqrt{\frac{\pi}{2}}}=\sqrt{\frac{2}{\pi}},$$
so that from \eqref{eq: Buser con cheeger raggiunta} it follows 
$$\Per(E) \leq\frac{1}{2}\sqrt{\frac{2}{\pi}}=\frac{1}{\sqrt{2\pi}}=I(1/2),$$
where $I$ is the Gaussian isoperimetric profile function. Since equality holds in the Buser inequality, equality holds in both the inequalities in \eqref{in: last part DeMo}. In particular $\mm(E)=1/2$.  Therefore $E$ achieves equality in the Gaussian isoperimetric inequality \eqref{eq:GII} and the conclusion is thus a consequence of Proposition \ref{prop: Han rigidity}.

\textbf{Case} (3): we claim that also this case can be ruled out. Indeed, if the supremum is attained for some $0<\bar{t}<\infty$, then by \eqref{in: last part DeMo} we get $\mm(E)=1/2$ and 
\begin{equation}
\left\Vert H_{\bar{t}/2}(\chi_E-\mm(E))\right\Vert^2_{2}=e^{-\lambda_1 \bar{t}}\left\Vert \chi_E-\mm(E)\right\Vert^2_{2}.
\end{equation}
Therefore $f:=\chi_E-\mm(E)$ attains the equality for $t=\bar{t}/2$ in the inequality
\begin{equation*}
\left\Vert H_t f\right\Vert_2\le e^{-\lambda_1t}\left\Vert f\right \Vert_2,
\end{equation*}
that has been proved in \cite[eq. (46)]{DeMo} through an application of Gronwall's lemma.
It follows that
\begin{equation*}
2\lambda_1\int_X|H_t f|^2\di\mm=2\int_X|\nabla H_t f|^2\di\mm=-\frac{\di}{\di t}\int_X|H_t f|^2\di\mm,
\end{equation*}
for a.e. $t\in(0,\bar{t})$ (cf. with \cite[eq. (47)]{DeMo}). Hence, by the classical characterization of the first eigenvalue of the Laplacian, $-\Delta H_tf=\lambda_1H_tf$ for a.e. $t\in(0,\bar{t})$. From this we infer by the heat equation that $H_tf=e^{-\lambda_1t}f$ for any $t\in(0,\infty)$. It follows by the regularizing properties of the heat semigroup that $f\in D(\Delta)$, in particular $f\in W^{1,2}(X,\di,\mm)$. 
We claim that this yields a contradiction. To do so it is sufficient to notice that $|\nabla f|= 0$ $\mm$-a.e. and to apply the local Poincaré's inequality \cite[Theorem 1]{Raj} to a recovery sequence of $\mathsf{Ch}_{\mm,2}(f)$ in order to obtain in the limit that for every $x\in X$ and every $r>0$
$$\int_{B(x,r)}|f-\langle f \rangle_{B(x,r)}|\,\di\mm=0.$$ By considering a ball $B(x,r)$ such that $\mm(B(x,r)\cap E)>0$ and $\mm(B(x,r)\cap E^c)>0$ we obtain that $f$ must be equal $\mm$-a.e. to its mean $\langle f \rangle_{B(x,r)}\in (-1/2,1/2)$ on $B(x,r)$, which contradicts the explicit expression of $f$. 

\end{proof}

\section{On the equality  case in Cheeger's inequality}

As pointed out in the introduction, Cheeger's inequality fits into a family of more general inequalities comparing eigenvalues of the $p$-Laplace operator for different exponents $p$.
To this regard, in the next proposition we extend to the general metric measure setting a recent result obtained in \cite{MR} in the context of Euclidean domains (see in particular Lemma 3.1 therein).\\ 
The statement and the argument appear to be well known to experts, we report them for the sake of completeness and since they will be relevant to investigate the rigidity later.
The proof is based on the strategy implemented in \cite{MR} for the case $p>1$, while in the case $p=1$ is based on \cite[Appendix A]{DeMo}.

\begin{proposition}\label{prop:mononeumann}
Let $(X,\di,\mm)$ be a metric measure space with $\mm(X)<\infty$. Then
\begin{equation}\label{eq:monNeu}
p\left(\lambda_{1,p}(X,\di,\mm)\right)^{\frac{1}{p}}\le q\left(\lambda_{1,q}(X,\di,\mm)\right)^{\frac{1}{q}}\,  \qquad\textrm{for every} \ \ 1\le p<q<\infty\, .
\end{equation}
\end{proposition} 

\begin{proof}
Let us consider separately the two cases $p>1$ and $p=1$.
\medskip  

\textbf{Case $p>1$.}
\\Fix $1<p<q$, $\varepsilon>0$ and choose $f\in\mathsf{Lip}_{bs}(X)$ such that 
$$\int_X |f|^q\, \di\mm=1\, , \quad \int_X |f|^{q-2}f\, \di\mm=0\, , \quad \int_X \mathrm{lip}(f)^q\, \di\mm\,\le \lambda_{1,q}+\varepsilon\, .$$
Let $S_q(X):=\{f\in L^q(X,\mm) : \|f\|_q=1\}$ endowed with the induced strong topology from $L^q(X,\mm)$ and define the continuous curve
$\gamma_q\in C(\mathbb{R}\,;\,S_q(X))$ as
$$\gamma_q(t):=\frac{f(\cdot)+t}{\|f(\cdot)+t\|_q}.$$
We also set $\gamma(t):=|\gamma_q(t)|^{(q-p)/p}\,\gamma_q(t)$, $t\in \mathbb{R}$.  Notice that $\gamma(t)\in S_p(X)$ for any $t\in\mathbb{R}$. Moreover, 
$$\lim_{t\to \pm \infty} \gamma(t)=\frac{\pm 1}{\mm(X)^{1/p}} ,$$
in the $L^p(X,\mm)$ sense. Hence by continuity one can easily infer the existence of $\tilde{t}\in \mathbb{R}$ such that 
$$\int_X |\gamma(\tilde{t})|^{p-2}\gamma(\tilde{t})\, \di\mm=0.$$

Using the trivial fact that for every $g\in \mathsf{Lip}_{bs}(X)$ and every $a,b\in \mathbb{R}$ it holds $\mathrm{lip}(ag+b)=|a|\mathrm{lip}(g)$, we have 
$$\| \mathrm{lip}(\gamma_q(t)) \|^q_q=\frac{\|\mathrm{lip}(f)\|_q^q}{\|f(\cdot)+t\|^q_q}, \quad t\in \mathbb{R}.$$
Since the function $t\mapsto \|f(\cdot)+t\|^q$ is minimized for $t=\int_X |f|^{q-2}f\, \di\mm=0$, we  deduce that the function $t\mapsto \| \mathrm{lip}(\gamma_q(t)) \|^q_q$ has a maximum at $t=0$ where it holds
$\| \mathrm{lip}(\gamma_q(0)) \|^q_q\le \lambda_{1,q}+\varepsilon.$
Thus, recalling the definition of $\lambda_{1,p}$ and using the H\"older's inequality we have:
\begin{align}\label{eq:almosteqeigen}
\nonumber\lambda_{1,p}&\le \max_{t\in\mathbb{R}}\| \mathrm{lip}(\gamma(t)) \|^p_p \le \left(\frac{q}{p}\right)^p\max_{t\in\mathbb{R}}\int_X|\gamma_q(t)|^{q-p}\, \mathrm{lip}(\gamma_q(t))^p\,\di\mm \\
&\le \left(\frac{q}{p}\right)^p\max_{t\in\mathbb{R}}\left(\int_X|\gamma_q(t)|^q\,\di\mm\right)^{(q-p)/p}\left(\int_X \mathrm{lip}(\gamma_q(t))^q\,\di\mm\right)^{p/q}\\
\nonumber &=\left(\frac{q}{p}\right)^p\max_{t\in\mathbb{R}}\left(\int_X\mathrm{lip}(\gamma_q(t))^q\,\di\mm\right)^{p/q}\le\left(\frac{q}{p}\right)^p\left(\lambda_{1,q}+\varepsilon\right)^{p/q}.
\end{align}
Since $\varepsilon >0$ is arbitrary, we get that for every $1<p<q$ 
$$\lambda_{1,p}(X,\di,\mm)\le \left(\frac{q}{p}\right)^p\left(\lambda_{1,q}(X,\di,\mm)\right)^{p/q}\, ,$$
which is \eqref{eq:monNeu}.
\vspace{3mm}

\textbf{Case $p=1$.}
\\By the very definition of $\lambda_{1,q}$, for every $\varepsilon>0$ there exists a non-null function $f\in \mathsf{Lip}_{bs}(X)$ with $\int_X |f|^{q-2}f\, \di\mm=0$ and
\begin{equation}\label{eq:PfCh1}
\lambda_{1,q}+\varepsilon\geq \frac{\int_X \mathrm{lip}(f)^q \, \di\mm}{\int_X |f|^q \, \di\mm} \, .
\end{equation}
We set $f^+:=\max \{f-m,0\}$ and $f^-:=-\min \{f-m,0\}$, where $m$ is any median of the function $f$. 
Applying the co-area inequality \eqref{eq:coarea} to $(f^+)^{q}$ (respectively $(f^-)^q$) and recalling the definition \eqref{eq:defChConst} of Cheeger's constant $h(X)$, we obtain
\begin{align}\label{eq:PfCh3}
&\int_X \mathrm{lip}[(f^+)^{q}] \, \di\mm + \int_X \mathrm{lip}[(f^-)^{q}] \, \di\mm  \\
&\geq\int_{0}^{\sup \{(f^+)^{q}\}}  \Per(\{(f^+)^{q}>t\}) \, \di t +\int_{0}^{\sup \{(f^-)^{q}\}}  \Per(\{(f^-)^{q}>t\}) \, \di t  \nonumber \\ 
&\geq h(X) \int_{0}^{\sup \{(f^+)^{q}\}}   \mm(\{(f^+)^{q}>t\}) \, \di t +h(X) \int_{0}^{\sup \{(f^-)^{q}\}}   \mm(\{(f^-)^{q}>t\}) \, \di t  \nonumber \\
&=h(X)\int_X  (f^+)^{q}  \di\mm+h(X) \int_X (f^-)^q  \di\mm = h(X)  \int_{X} |f-m|^{q} \, \di\mm\, . \nonumber
\end{align}
Observe that, for any nonnegative function $g$, 
$$\mathrm{lip}[g^{q}]\leq q |g|^{q-1} \mathrm{lip}(g)\, ,$$
and 
$$\mathrm{lip}(f^+)\leq \mathrm{lip}(f), \ \  \mathrm{lip}(f^-)\leq \mathrm{lip}(f)\, .$$
By applying H\"older's inequality we get
\begin{equation}\label{eq:PfCh4}
q\bigg( \int_X \mathrm{lip}(f)^q \, \di\mm \bigg)^{\frac{1}{q}}\bigg( \int_X |f-m|^q \, \di\mm \bigg)^{1-\frac{1}{q}}\geq \int_X \mathrm{lip}[(f^+)^{q}] \, \di\mm + \int_X \mathrm{lip}[(f^-)^{q}] \, \di\mm\, ,
\end{equation}
where we have used that  $|f^+| + |f^-|=|f-m|$.
It follows from \eqref{eq:PfCh3} and \eqref{eq:PfCh4} that for every median $m$ of $f$ it holds
\begin{equation}
\frac{\int_X \mathrm{lip}(f)^q \, \di\mm}{\int_X |f-m|^q \, \di\mm}\geq \left(\frac{h(X)}{q}\right)^q.
\end{equation}
Finally, since $m_q(f):=\int_X |f|^{q-2}f \di\mm=0$ and $m_q$ minimises $\R\ni c\mapsto \int_X |f-c|^q \di\mm$, we have
$$\frac{\int_X \mathrm{lip}(f)^q \, \di\mm}{\int_X |f|^q \, \di\mm}\geq \frac{\int_X \mathrm{lip}(f)^q \, \di\mm}{\int_X |f-m|^q \, \di\mm} $$ 
and we can conclude thanks to \eqref{eq:PfCh1} and the fact that $\varepsilon>0$ is arbitrary.
\end{proof}

\begin{remark}
Let us mention that the monotonicity result above still holds, in very general frameworks, if one replaces the Neumann eigenvalues with \emph{Dirichlet} eigenvalues, see \cite[Theorem 3.2]{Li} for the case of Euclidean domains. 
\end{remark}

It turns out that the monotonicity  in \eqref{eq:incr} is strict, under general assumptions. This was pointed out for Euclidean domains and Dirichlet eigenvalues in \cite{Li} (see the Remark after the proof of Theorem 3.2 therein); the strategy proposed can be adapted to our framework. In the next theorem we explore some instances of this phenomenon, restricting for the sake of simplicity to the framework of $\mathsf{RCD}$ spaces. Let us point out, however, that we expect this rigidity \emph{not} to be linked with a specific synthetic Ricci curvature lower bound, nor with the infinitesimally Hilbertian assumption, even though some \emph{regularity} assumption is necessary, as the examples of Section \ref{subsec:examples} will illustrate.
\medskip

\begin{theorem}\label{thm:strictNeum}
Let $(X,\di,\mm)$ be an $\mathsf{RCD}(K,\infty)$ metric measure space with $\mm(X)<\infty$. Let $q>1$ and let us suppose that there exists a function $f\in \Lambda_q(X,\di,\mm)$ that minimizes \eqref{eq: eqdeflambda1p}.
Then, for every $1\le p<q$ it holds
\begin{equation}\label{eq:strictMonpq}
p\left(\lambda_{1,p}(X,\di,\mm)\right)^{\frac{1}{p}}<q\left(\lambda_{1,q}(X,\di,\mm)\right)^{\frac{1}{q}}\, .
\end{equation} 
\end{theorem}

To prove the strict monotonicity we will rely on the following technical lemma. Basically it amounts to say that a non-null Sobolev function $f$ such that $|\nabla f|\le C|f|$ for some constant $C\ge 0$ needs to have constant sign and cannot vanish on a large set.

\begin{lemma}\label{lemma:techngrad}
Let $(X,\di,\mm)$ be an $\mathsf{RCD}(K,\infty)$ metric measure space and let $p>1$. Assume that there exists a function $f\in W^{1,p}(X,\di,\mm)$ such that
\begin{equation}\label{eq:asscontr}
|\nabla f|\le C|f|\, ,\quad\text{$\mm$-a.e. on $X$}\, ,
\end{equation}
for some constant $C\ge 0$. Then, either $f>0$, $f=0$ or $f<0$ holds $\mm$-a.e. on $X$.
\end{lemma}

\begin{proof}
Let $f\in W^{1,p}(X,\di,\mm)$ be a non-null function and let us suppose by contradiction and without loss of generality that $\mm(\{f>0\})>0$ and $\mm(\{f\le 0\})>0$. In particular, we can find two bounded sets $A_1,A_2\subset X$ with positive and finite measure and such that $f>0$ a.e. on $A_1$ and $f\le 0$ a.e. on $A_2$. Let $\mu_1$ and $\mu_2$ be the probability measures obtained by restriction and normalization of the measure $\mm$ to $A_1$ and $A_2$ respectively. Observe that $f> 0$ holds $\mu_1$-a.e. and $f\le 0$ holds $\mu_2$-a.e. .\\ 
Then let $\Pi\in\mathcal{P}(C([0,1],X))$ be the optimal geodesic plan representing the $W_2$-geodesic between $\mu_1$ and $\mu_2$. Observe that $\Pi$ is a test plan and it is concentrated on constant speed geodesics. Therefore the following conditions are satisfied for $\Pi$-a.e. $\gamma\in C([0,1],X)$:
\begin{itemize}
\item[(i)] $f\circ\gamma:[0,1]\to\mathbb{R}$ has an absolutely continuous and $W^{1,p}$ representative, that we shall identify with $f\circ\gamma$ without risk of confusion;
\item[(ii)] $f(\gamma(0))>0$ and $f(\gamma(1))\le 0$;
\item[(iii)] $|\frac{\di}{\di t}f(\gamma(t))|\le C_{\gamma}|f(\gamma(t))|$ for a.e. $t\in (0,1)$, for some constant $C_{\gamma}\ge 0$ (depending on $C$ and the length of the geodesic $\gamma$).
\end{itemize}
Conditions (i) and (iii) follow from the fact that $\Pi$ is a test plan, together with \eqref{eq:asscontr}. Condition (ii) follows from the fact that $f>0$ and $f\le 0$ hold $\mu_1$ and $\mu_2$-a.e., respectively, and $({\rm e}_0)_{\sharp}\Pi=\mu_1$, $({\rm e}_1)_{\sharp}\Pi=\mu_2$.

Let us consider a curve $\gamma$ such that (i), (ii) and (iii) are verified. We next prove that this yields to a contradiction.  Indeed, letting $t_0\in(0,1]$ be such that $f(\gamma(t_0))=0$, by integrating (iii) we easily obtain the inequality
\begin{equation}
|f(\gamma(t))|\le \int_{[t_0,t]}C_{\gamma}|f(\gamma(s))|\di s\, ,
\end{equation} 
for any $t\in[0,1]$. The integral form of Gronwall's inequality yields then that $|f(\gamma(t))|=0$ for any $t\in[0,1]$, contradicting (ii).
\end{proof}

\begin{proof}[Proof of Theorem \ref{thm:strictNeum}]
Fix $1<p<q$ and let us suppose by contradiction that 
\begin{equation}\label{eq:equalityneum}
p\left(\lambda_{1,p}(X,\di,\mm)\right)^{\frac{1}{p}}=q\left(\lambda_{1,q}(X,\di,\mm)\right)^{\frac{1}{q}}\, .
\end{equation}
By assumption, there exists $f\in W^{1,q}(X,\di,\mm)$ such that
\begin{equation}\label{eq:condneumeig}
\int_X|f|^q\, \di\mm=1\, ,\,\,\,\;\int_X|f|^{q-2}f\, \di\mm=0\,,\,\;\,\,\int_X|\nabla f|^q\, \di\mm=\lambda_{1,q}(X,\di,\mm)\, .
\end{equation}
Notice that we can apply the very same argument used in the proof of Proposition \ref{prop:mononeumann} (case $p>1$) with $\varepsilon=0$ and by replacing the slope with the minimal weak upper gradient. Since equality holds in \eqref{eq:equalityneum}, equality holds in H\"older's inequality, that we used in equation \eqref{eq:almosteqeigen}. In particular, there exists a constant $C\ge 0$ such that $|f|=C|\nabla f|$ $\mm$-a.e.

We now appeal to Lemma \ref{lemma:techngrad} to reach the desired contradiction, since by the first two conditions in \eqref{eq:condneumeig} we know that $f$ is non-null and must change its sign.
\medskip

Finally,  notice that \eqref{eq:strictMonpq} holds also for $p=1$: indeed the strict inequality trivially follows by the case $p>1$ and Proposition \ref{prop:mononeumann}.

\end{proof}
An immediate corollary is the following:
\begin{corollary}\label{cor:opteqCheeg}
Let $(X,\di,\mm)$ be an $\mathsf{RCD}(K,\infty)$ metric measure space, with $\mm(X)=1$,  admitting a superlinear isoperimetric profile. Then the function
\begin{equation}\label{eq:incr}
[1,\infty)\ni p\mapsto p\left(\lambda_{1,p}(X,\di,\mm)\right)^{\frac{1}{p}}
\end{equation}
is strictly increasing.
\end{corollary}
\begin{proof}
Let us observe that the infimum in the variational definition of $\lambda_{1,p}$ is attained for any $1<p<\infty$, under our assumptions. This is a consequence of the superlinearity of the isoperimetric profile which yields in turn the compactness of the embedding of $W^{1,p}$ into $L^p$, see Theorem \ref{thm:compembedding} (i).
The result thus follows from \autoref{thm:strictNeum}.
\end{proof}

Under some additional assumptions, we show that the Cheeger's inequality is strict even when $\mm(X)=\infty$. Let us focus on the case $p=1$ and $q=2$, and notice that here $\lambda_1$ is naturally replaced by $\lambda_0$. 
\begin{theorem}\label{thm:rigcheeginfinite}
Let $(X,\di,\mm)$ be an $\mathsf{RCD}(K,\infty)$ metric measure space with $\mm(X)=\infty$ and $K\in\R$. Let us suppose that $\lambda_0:=\inf{\sigma_d(-\Delta)}<\Sigma$ (this is always the case if the spectrum is discrete) and that the $\lambda_0$-eigenfunction is in $L^{\infty}(X,\mm)$.

Then the equality in Cheeger's inequality is never attained, i.e.
\begin{equation}\label{eq:ChIneStrictinfinite}
\lambda_0>\frac{1}{4}h(X)^2.
\end{equation}
\end{theorem}
\begin{proof}
The proof is by contradiction. 

\textbf{Step 1.} Aim for this step is to prove that, assuming equality holds in Cheeger's inequality, we obtain the existence of a function $f\in D(\Delta)$ such that $-\Delta f=\lambda_0 f$ and 
\begin{equation}\label{eq:modgrad}
|\nabla f|=\sqrt{\lambda_0}\,|f|,\,\,\,\text{$\mm$-a.e. .}
\end{equation}

First of all, under our assumptions \autoref{thm:compembedding} implies that $\lambda_0$ (and thus also $h(X)$) is positive and that there exists $f\in D(\Delta) \subset  W^{1,2}(X,\di,\mm)$ with $-\Delta f= \lambda_{0}f$, and $\|f\|_2=1$.  
From Lemma \ref{lem:fW12f2BV} we know that
$$|D (f)^{2}|(X)\leq 2 \int_{X}  |f|\, | \nabla f| \, \di \mm, $$
and we can apply the Cauchy-Schwarz inequality to infer

\begin{equation}\label{eq:CSProof}
2\left( \int_X |\nabla f|^2 \, \di \mm \right)^{\frac{1}{2}}\left( \int_X |f|^2 \, \di \mm \right)^{\frac{1}{2}}\geq |D(f)^{2}|(X) 
\end{equation}
Using the coarea formula \eqref{eq:coareaBV} and recalling that $\|f\|_2=1$, we obtain
\begin{equation}\label{eq:almosteqCh}
\begin{aligned}
& 2\sqrt{\lambda_0}=2\bigg( \int_X | \nabla f|^2 \, \di\mm \bigg)^{\frac{1}{2}} 
\overset{\eqref{eq:CSProof}}{\geq} |D(f)^{2}|(X)   \overset{ \eqref{eq:coareaBV}} {=}  \int_{0}^{\infty}  \Per(\{(f)^{2}>t\}) \, \di t   \\ 
&\ge h(X)\int_{0}^{\infty}   \mm(\{(f)^{2}>t\}) \, \di t  
=    h(X)\int_{X} f^{2} \, \di \mm  =h(X)=2\sqrt{\lambda_0}, 
\end{aligned}
\end{equation}
where the last identity comes from the assumption that equality is achieved in Cheeger's inequality. It follows that all the inequalities in \eqref{eq:almosteqCh} are actually equalities. In particular 
\begin{equation}\label{eq:CSeq}
\bigg( \int_X |\nabla f|^2 \, \di\mm \bigg)^{\frac{1}{2}}=\int_X |\nabla f||f|\,\di\mm\, .
\end{equation}
\\By the equality case \eqref{eq:CSeq} in the Cauchy-Schwartz inequality we can infer that $|\nabla f|=C|f|$ $\mm$-a.e., for some constant $C> 0$. Since
$$\lambda_0=\int_X |\nabla f|^2\,\di\mm=\int_X C^2f^2\,\di\mm=C^2$$
the claim \eqref{eq:modgrad} follows.
\medskip

\textbf{Step 2.}
Since $f\in L^{\infty}(X,\mm)$ and $-\Delta f=\lambda_0f$ we have that $\Delta f\in L^{\infty}(X,\mm)$ and also (up to a suitable choice of $\mm$-a.e. representative) $f\in \mathsf{Lip}_b(X)$ by the $L^{\infty}-\mathsf{Lip}$ regularization of the heat semigroup \eqref{strong Feller}.

We are now in position to apply \cite[Theorem 9.6 (b)]{AT14}. Thus, for any $T>0$ we can find a test plan $\Pi\in\mathcal{P}(C([0,T],X))$ such that $({\rm e}_0)_{\sharp}\Pi=\mm$ and 
\begin{equation}\label{eq:flow}
f(\gamma(t))-f(\gamma(s))=\int_s^t|\nabla f|^2(\gamma(r))\di r=\int_s^t\lambda_0f^2(\gamma(r))\di r,
\end{equation}
for any $0\le s\le t\le T$ and for $\Pi$-almost every $\gamma$. Observe that, fixing any such curve $\gamma$ and setting $F(t):=f(\gamma(t))$, $F$ is smooth and it verifies the ODE
\begin{equation}\label{eq:ode}
F'(t)=\lambda_0F^2(t)
\end{equation}
in the classical sense. Indeed it is absolutely continuous, it solves the ODE in the almost everywhere sense and the derivative itself is continuous, allowing to bootstrap the regularity.
Observe that if $F(0)=0$, then $F(t)=0$ for every $t\in[0,T]$. Otherwise, if $F(0)>0$ there is no bounded solution of \eqref{eq:ode} up to time $1/(\lambda_0F(0))$.

We claim that $f$ vanishes identically, contradicting the assumption $\int_X f^2\, \di\mm=1$. If this is not the case we can find $\epsilon>0$ and a set of positive measure $E\subset X$ such that for any $x\in E$ it holds $f(x)\lambda_0>\epsilon$. Then we apply the construction above with $T=1/\epsilon$ finding $\Pi\in\mathcal{P}(C([0,T],X))$ verifying \eqref{eq:flow} and such that, for a set of positive $\Pi$-measure of curves $\gamma$, it holds $f(\gamma(0))\lambda_0>\epsilon$. Recalling what we pointed out above concerning bounded solutions of \eqref{eq:ode}, we obtain a contradiction, since $f$ is bounded.

\end{proof}

\begin{remark}\label{rem:CommentsThm41}
We notice that there exist $\mathsf{RCD}(K,\infty)$ metric measure spaces satisfying the assumptions of \autoref{thm:rigcheeginfinite}. For instance, let us consider $\mathbb{R}$ endowed with the Euclidean distance $\mathsf{d}(x,y)=|x-y|$ and the measure $\mm:=e^{x^2/2}d\mathcal{L}^1.$ One can easily see that it satisfies the $\mathsf{RCD}(-1,\infty)$ condition. A result of Wang \cite[Example 5.1]{Wang} ensures that the spectrum is discrete.  Finally, by using Proposition \ref{prop:Tamaninicond}, the associated heat semigroup is ultracontractive (thus any eigenfuction of the $2$-Laplacian is in $L^{\infty}$). 

\end{remark}

\medskip

\subsection{Examples}\label{subsec:examples}

We have seen in the previous results that the existence of a first eigenfunction of the Laplacian is a relevant assumption in order to obtain that the Cheeger's inequality is strict. We now collect a series of examples of  $\mathsf{RCD}$ spaces (actually, smooth Riemannian manifolds) where this assumption is not satisfied and equality in Cheeger's inequality is achieved.

\begin{example}\label{example:eqtriv0}
\textbf{An $\mathsf{RCD}(0,n)$ space with infinite measure satisfying} $\lambda_0(X)=h(X)=0$.

A classical example of equality in Cheeger's inequality is obtained in the Euclidean space $(\R^n,|\cdot|,\di\mathcal{L}^n)$. Indeed, it is well known that $\lambda_0(\R^n)=0$ (see e.g. \cite[Example 10.9]{Grig}). By Cheeger's inequality and the trivial nonnegativity of the Cheeger constant (or by direct computation, considering balls of increasing radii as competitors in the definition), we also have $h(\R^n)=0$. 
\end{example}

\begin{example}\label{example:eqtriv1}
\textbf{An $\mathsf{RCD}(K,2)$ space with finite measure satisfying } $\lambda_1(X)=h(X)=0$.

The following example is strongly inspired by \cite{Br}. For $|t|>1$, let us define the function $f(t):=e^{-\sqrt{|t|}}$. We consider an extension $F:\R\rightarrow \R$ of $f$ such that $F\in \mathcal{C}^{\infty}(\R)$, $F$ is even and $F(t)>0$ for every $t\in \R$. 
We denote by $S$ the surface of revolution parametrized by 
$$(t,\theta)\mapsto (F(t)\cos(\theta),F(t)\sin(\theta),t)\, , \qquad (t,\theta)\in \R\times [0,2\pi).$$
The surface $S$ is a $2$-dimensional Riemannian manifold with the warped product structure $\R\times_F \mathbb{S}^1$, where the $\R$ factor is endowed with the arc-length metric $dx^2=(1+F'(t)^2)dt^2$ and the $\mathbb{S}^1$ factor  is endowed with the standard metric. 

We claim that $S$ has finite volume, Gaussian curvature bounded from below and $\lambda_1(X)=0$.

Indeed, 
$$\mathrm{vol}(S)=2\pi\int_{-\infty}^{\infty}F(t)\sqrt{1+(F'(t))^2}\,\di t<\infty\, ,$$
since the integrand is continuous and integrable at infinity as a consequence of the asymptotic
$$F(t)\sqrt{1+(F'(t))^2}\sim e^{-\sqrt{|t|}}\qquad \textrm{as} \ |t|\to \infty\, .$$

The Gaussian curvature $\mathsf{K}$ can be computed using a classical formula for surfaces of revolution, i.e.
$$\mathsf{K}_S(t,\theta)=-\frac{F''(t)}{F(t)\sqrt{1+(F'(t))^2}}\, .$$
We thus observe that $\mathsf{K}_S$ is bounded from below since $F$ is smooth, strictly positive and 
$$-\frac{F''(t)}{F(t)\sqrt{1+(F'(t))^2}}=-\frac{\sqrt{|t|}+1}{4|t|^{3/2}\sqrt{1+\frac{e^{-2\sqrt{|t|}}}{4|t|}}}>-\frac{1}{2}\, ,    \qquad \textrm{for} \ |t|>1\, . $$
Thus $S$ is an $\mathsf{RCD}(K,2)$ space, for some $K\in \R$.
It remains to show that $\lambda_1(S)=0$ (which will imply in turn that $h(S)=0$ by Cheeger's inequality). In order to prove this, it is sufficient to show (see also \cite[Section $3$]{Ledoux})
$$\mu_S:=\limsup_{r\to \infty} \frac{1}{r}\log\Big(\mathrm{vol}(S)-\mathrm{vol}(B_r)\Big)=0\, ,$$
where $B_r$ denotes the geodesic ball of centre $(0,0)$ and radius $r$.

Since by elementary considerations 
$$\mathrm{vol}(S)-\mathrm{vol}(B_r)\geq 2\pi\int_{|t|>x}F(t)\sqrt{1+(F'(t))^2}\,\di t$$
where $x$ is defined so that
$$r=\int_0^x \sqrt{1+(F'(t))^2}\,\di t\, ,$$
we obtain
$$\mu_S\geq \limsup_{x\to \infty} \frac{\log\Big(\int_{x}^{\infty}F(t)\sqrt{1+(F'(t))^2}\,\di t\Big)}{\int_0^x \sqrt{1+(F'(t))^2}\,\di t}$$
and the limit superior in the right hand side is actually a limit equal to $0$. To see this, one can apply twice L'Hospital's rule and then use the explicit expression of $F(x)$ for $x>1$. 
Since it trivially holds that $\mu_S\leq 0$, we have $\mu_S=0$ and the claim follows.

\end{example}

\begin{example}
\textbf{An $\mathsf{RCD}(-1,2)$ space with infinite measure satisfying }$\lambda_0(X)=\frac{1}{4}h(X)^2>0$.

We claim that the hyperbolic plane $\mathbb{H}^2$ realizes the equality in Cheeger's inequality with the additional property, with respect to Example \ref{example:eqtriv0}, that the bottom of the spectrum and the Cheeger constant are non trivial. Of course, the volume of the hyperbolic plane is infinite.

Let us recall that, as proved for instance in \cite{Mc}, on the hyperbolic plane (equipped with the canonical volume measure) it holds $\lambda_0=1/4$.

Let us verify that the Cheeger constant of the hyperbolic plane equals $1$. In order to do so we recall the isoperimetric inequality 
\begin{equation}\label{eq:isophyp}
\Per(A)^2\ge 4\pi \mm(A)+\mm(A)^2,
\end{equation}
for any set of finite perimeter $A\subset\mathbb{H}^2$, see \cite{Be,Os,Sc} dealing with sets with smooth boundary, the extension to sets of finite perimeter can be obtained with standard approximation arguments. Moreover, we recall that geodesic balls realize the equality in \eqref{eq:isophyp}. From \eqref{eq:isophyp} we easily deduce that
\begin{equation*}
\frac{\Per(A)}{\mm(A)}\ge 1,
\end{equation*}
for any $A\subset\mathbb{H}^2$ with finite perimeter. This proves that $h(\mathbb{H}^2)\ge 1$. To prove that $h(\mathbb{H}^2)=1$ we just observe that geodesic balls with radii going to infinity verify
\begin{equation*}
\frac{\Per(B_r)}{\mm(B_r)}\to 1,
\end{equation*}
as $r\to \infty$, by direct computation or by equality in \eqref{eq:isophyp}. This proves that $h(\mathbb{H}^2)=1$ and therefore equality holds in Cheeger's inequality.

\end{example}

\begin{example}\label{examplethrice}
\textbf{An $\mathsf{RCD}(-1,2)$ space with finite measure satisfying } $\lambda_1(X)=\frac{1}{4}h(X)^2>0$.

We claim that an example of (actually smooth) metric measure space with finite reference measure, verifying the $\mathsf{RCD}(-1,2)$ condition and the equality in Cheeger's inequality is given by the symmetric three-punctured sphere with hyperbolic metric, that we shall denote by $\mathbb{D}$. Moreover in this case 
\begin{equation}\label{eq:isotriang}
\lambda_1(\mathbb{D})=\frac{1}{4}h(\mathbb{D})^2>0\, .
\end{equation} 

The example is strongly inspired by \cite{Buser79} where sharpness of the Cheeger inequality was pointed out exhibiting a family of compact Riemannian manifolds almost attaining the inequality.

Let us briefly recall how a hyperbolic metric on the three-punctured sphere can be built, referring to \cite[Section 10.5]{Pri} for a more detailed construction and all the relevant background on hyperbolic geometry.\\ 
This hyperbolic manifold can be seen as a degenerate pair of hyperbolic pants, with cusps in place of the three boundary components. More in detail we can also obtain it considering a degenerate hexagon on the Poincar\'e disk model of the hyperbolic plane (i.e. we consider three points on the boundary of the disk equidistant with respect to the standard metric and connect them with hyperbolic geodesics) and gluing it with itself along the three boundary components (i.e. we consider the double of the starting triangle $\mathbb{T}$). Observe that the resulting Riemannian manifold, that we shall denote by $\mathbb{D}$, is a non compact, complete hyperbolic manifold. In particular it has constant sectional curvature $-1$ and therefore it is an $\mathsf{RCD}(-1,2)$ metric measure space when endowed with the canonical volume measure $\mathrm{vol}$. 

We claim that $\mathbb{D}$ has finite volume, in particular it holds that $\mathrm{vol}(\mathbb{D})=2\pi$. We just provide a sketch of the strategy to verify this conclusion, since the result is well known.

The more direct way to check this conclusion is by directly computing $\mathrm{vol}(\mathbb{T})=\pi$, using the explicit formulas for the Poincar\'e disk model, and then to argue that $\mathrm{vol}(\mathbb{D})=2\pi$, since $\mathbb{D}$ is the double of $\mathbb{T}$. Alternatively one can rely on a general version of Gauss-Bonnet formula \cite{Hu} taking into account the fact that $\mathbb{D}$ is homeomorphic to the sphere with three punctures and therefore it has Euler characteristic $\chi(\mathbb{D})=-1$. Therefore, denoting by $\mathsf{K}_{\mathbb{D}}$ the Gaussian curvature,
\begin{equation*}
-\mathrm{vol}(\mathbb{D})=\int_{\mathbb{D}} \mathsf{K}_{\mathbb{D}}\,\di\mathrm{vol}=2\pi\chi(\mathbb{D})=-2\pi.
\end{equation*}

We divide the verification of \eqref{eq:isotriang} in two steps.\\ 
First let us prove that $h(\mathbb{D})=1$. In order to do so we rely on the study of the isoperimetric problem on hyperbolic surfaces pursued in \cite{AdM}. Since $\mathbb{D}$ has three cusps (corresponding to the three punctures of the sphere), by the last part of the statement of \cite[Theorem 2.2]{AdM} (see also the remark after its proof) we get that, for any value of the area $0<v\le \pi=\mathrm{vol}(\mathbb D)/2$, it holds that 
\begin{equation*}
\Per(A)\ge \mathrm{vol}(A)=v,
\end{equation*}
for any set of finite perimeter $A$ such that $\mathrm{vol}(A)=v$. Moreover there are sets for which equality is attained in the above inequality (neighbourhoods of cusps bounded by horocycles). Therefore, by the very definition of the Cheeger constant, it holds $h(\mathbb{D})=1$.

We are thus left with the verification of the identity $\lambda_1(\mathbb{D})=1/4$. Observe that thanks to Cheeger's inequality it is sufficient to prove that $\lambda_1(\mathbb{D})\le 1/4$, the other inequality will follow from our estimate on the Cheeger constant. In order to do so we exhibit a sequence of Lipschitz functions $f_n:\mathbb{D}\to\R$ such that
\begin{equation*}
\int_{\mathbb{D}} f_n\,\di\mathrm{vol}=0,\;\;\;\int_{\mathbb{D}} f_n^2\,\di\mathrm{vol}=1,
\end{equation*}
for any $n\in\mathbb{N}$ and
\begin{equation*}
\int_{\mathbb{D}}|\nabla f_n|^2\,\di\mathrm{vol}\to 1/4,\;\;\;\text{as $n\to\infty$}\, .
\end{equation*}
The conclusion $\lambda_1(\mathbb{D})\le 1/4$ will follow from \eqref{eq:defla1Intro}.

Let us denote by $\lambda_1^D(\Omega)$ the first Dirichlet eigenvalue of the Laplacian on a smooth domain $\Omega$ contained in a Riemannian manifold. Recall that there is a variational characterization for $\lambda_1^D$ analogous to \eqref{eq:defla1Intro}.\\
As we already observed, $\mathbb{D}$ has constant Gaussian curvature $-1$. Therefore Cheng's inequality \cite[Theorem 1.1]{Che} applies and yields that for any $x\in\mathbb{D}$ and for any $r>0$ it holds
\begin{equation}\label{eq:cheng}
\lambda_1^D(B^{\mathbb{D}}_r(x))\le \lambda_1^D(B_r^{\mathbb{H}^2}(\bar{x}))\, ,
\end{equation}
where $B_r^{\mathbb{H}^2}(\bar{x})$ is the ball of radius $r$ and centre $\bar{x}$ in the hyperbolic plane. Moreover it is known (see for instance the top of \cite[p. 294]{Che}) that
\begin{equation}\label{eq:boundaboveeig}
\lambda_1^D(B_r^{\mathbb{H}^2}(\bar{x}))\le \frac{1}{4}+\left(\frac{2\pi}{r}\right)^2,
\end{equation}
for any $r>0$.
Combining \eqref{eq:cheng} with \eqref{eq:boundaboveeig} we infer that for any $\epsilon>0$ there exists $r>0$ such that for any $x\in\mathbb{D}$ it holds 

\begin{equation}\label{eq:secondbound}
\lambda_1^D(B^{\mathbb{D}}_r(x))\le \frac{1}{4}+\epsilon.
\end{equation}
Next we choose points $x_1,x_2\in\mathbb{D}$ such that $\di(x_1,x_2)>2r$, where we denoted by $\di$ the Riemannian distance induced by the hyperbolic metric on $\mathbb{D}$. By \eqref{eq:secondbound} and the variational characterization of the first Dirichlet eigenvalue we can find non negative Lipschitz functions $f_1^{\epsilon},f_2^{\epsilon}$ with compact support in $B_r(x_1)$ and $B_r(x_2)$ respectively and such that 
\begin{equation}
\int_{\mathbb{D}}(f_1^{\epsilon})^2\,\di\mathrm{vol}=\int_{\mathbb{D}}(f_2^{\epsilon})^2\,\di\mathrm{vol}=1
\end{equation}
and 
\begin{equation}
\int_{\mathbb{D}}|\nabla f_1^{\epsilon}|^2\,\di\mathrm{vol}\le \frac{1}{4}+\epsilon,\;\;\;\int_{\mathbb{D}}|\nabla f_2^{\epsilon}|^2\,\di\mathrm{vol}\le \frac{1}{4}+\epsilon.
\end{equation}
Next we observe that we can find coefficients $a_1^{\epsilon},a_2^{\epsilon}\in\R$ such that, setting $f^{\epsilon}:=a_1^{\epsilon}f_1^{\epsilon}+a_2^{\epsilon}f_2^{\epsilon}$, it holds
\begin{equation*}
\int_{\mathbb{D}}f^{\epsilon}\,\di\mathrm{vol}=0,\;\;\; \int_{\mathbb{D}}(f^{\epsilon})^2\,\di\mathrm{vol}=1
\end{equation*}
and 
\begin{equation*}
\int_{\mathbb{D}}|\nabla f^{\epsilon}|^2\,\di\mathrm{vol}\le \frac{1}{4}+\epsilon.
\end{equation*}
Since $f^{\epsilon}$ is an admissible competitor in the variational definition of $\lambda_1(\mathbb{D})$ and $\epsilon$ is arbitrary we infer that $\lambda_1(\mathbb{D})\le 1/4$, as desired.
\end{example}

\end{document}